\documentclass[a4paper,oneside,11pt]{amsart}

\reversemarginpar

\usepackage[colorlinks,hyperindex,linkcolor=blue,urlcolor=black,pdftitle={Examples of cyclic polynomially bounded operators that are not similar to contractions},pdfauthor={Maria F.Gamal'}]
{hyperref}

\usepackage{amsmath}
\usepackage{amsfonts}
\usepackage{amssymb}
\usepackage{amsthm}

\usepackage[english]{babel}
\usepackage{enumerate}

\usepackage{graphicx}
\usepackage{color}

\usepackage[abbrev]{amsrefs}

\numberwithin{equation}{section}

\theoremstyle{plain}
\newtheorem{theorem}{Theorem}[section]
\newtheorem{lemma}[theorem]{Lemma}
\newtheorem{corollary}[theorem]{Corollary}
\newtheorem{proposition}[theorem]{Proposition}

\theoremstyle{definition}
\newtheorem{remark}[theorem]{Remark}

\begin{document}

\title[Examples of cyclic polynomially bounded operators]{Examples of cyclic polynomially bounded operators that are not similar to contractions}

\author{Maria F. Gamal'}
\address{
 St. Petersburg Branch\\ V. A. Steklov Institute 
of Mathematics\\
 Russian Academy of Sciences\\ Fontanka 27, St. Petersburg\\ 
191023, Russia  
}
\email{gamal@pdmi.ras.ru}


\subjclass[2010]{ Primary 47A60; Secondary 47A65,  47A16, 47A20.}

\keywords{Polynomially bounded operator, similarity, contraction, 
unilateral shift, isometry, $C_0$-contraction, $C_0$-operator}


\begin{abstract}
A question if a polynomially bounded 
operator is similar  
to a contraction was posed by Halmos and was answered in the negative 
by Pisier.  His counterexample is an operator
of infinite multiplicity, 
while all
its restrictions on invariant subspaces of finite multiplicity are similar to 
contractions.  In the paper, cyclic polynomially bounded 
operators which are 
not similar to  contractions and are quasisimilar to $C_0$-contractions 
or to isometries 
are constructed. The construction is based on a perturbation of the sequence
 of finite dimensional operators which is uniformly polynomially bounded,
 but is not uniformly completely polynomially bounded, constructed by Pisier.
\end{abstract} 

\maketitle

\section{Introduction}

Let $\mathcal H$ be a Hilbert space, and let $T\colon \mathcal H\to \mathcal H$ 
be an operator
on $\mathcal H$.
$T$ is called {\it polynomially bounded}, if there exists a constant $M$
such that 
 \begin{equation} \label{11} \|p(T)\|\leq M\|p\|_\infty \text{ \ for every polynomial } p. 
\end{equation}
The smallest constant $M$ which satisfies  \eqref{11} is denoted by $M_{pb}(T)$ 
and is 
called the {\it polynomial bound} of $T$. If $T$ is not polynomially bounded,
then $M_{pb}(T) = \infty$.

For a natural number $n$ a $n\times n$ matrix can be regarded as an operator on 
$\ell_n^2$, its norm is denoted by the symbol $\|\cdot\|_{\mathcal L(\ell_n^2)}$. 
For a family of polynomials $[p_{ij}]_{i,j=1}^n$ put
$$\| [p_{ij}]_{i,j=1}^n\|_{H^\infty(\ell_n^2)}= 
\sup\{\| [p_{ij}(z)]_{i,j=1}^n\|_{\mathcal L(\ell_n^2)}, \ z\in\mathbb  D\}.$$
For an operator $T$ on $\mathcal H$ and a family of polynomials 
$[p_{ij}]_{i,j=1}^n$ 
the operator $[p_{ij}(T)]_{i,j=1}^n$ acting on the space 
$\oplus_{j=1}^n \mathcal H$ 
is defined.
$T$ is called {\it completely polynomially bounded}, 
if there exists a constant $M$ such that
 \begin{equation} \label{12}\begin{aligned}\| [p_{ij}(T)]_{i,j=1}^n\|\leq
M\| [p_{ij}]_{i,j=1}^n\|_{H^\infty(\ell_n^2)}\\ 
\text{ \ for every family of polynomials } [p_{ij}]_{i,j=1}^n.  \end{aligned}\end{equation}
The smallest constant $M$ which satisfies  \eqref{12} is denoted by $M_{cpb}(T)$ and is 
called the {\it complete polynomial bound} of $T$. If $T$ is not complete 
polynomially bounded, then $M_{cpb}(T) = \infty$.

An operator $T$ is called a {\it contraction} if $\|T\|\leq 1$.
The following criterion for an operator to be similar to a contraction is proved 
in [Pa]:

{\it An operator $T$ is similar to a contraction if and only if 
$M_{cpb}(T)<\infty$,
and $$M_{cpb}(T) = 
\inf \{\|X\|\|X^{-1}\|: \|XTX^{-1}\| \leq 1\},$$ 
where the operators $X$ are invertible.}

The question if a polynomially bounded 
operator is similar  
to a contraction was posed by Halmos [H] and was answered in the negative 
by Pisier [Pi].  The counterexample is an appropriate Foguel--Hankel operator. 
Recall that $\Gamma$ is a (vector-valued) Hankel operator if and only if
 $S_k^\ast\Gamma = \Gamma S_n$, where $1\leq n, k\leq\infty$, 
$S_n$ is the unilateral shift of multiplicity $n$. 
A Foguel--Hankel operator is an operator of the form 
$$\begin{pmatrix} S_k^\ast & \Gamma \\  \mathbb  O & S_n\end{pmatrix},$$
where $1\leq n, k\leq\infty$, $\Gamma$ is a Hankel operator.
There exists a Hankel operator $\Gamma$ such that the operator   
$$T=\begin{pmatrix} S_\infty^\ast & \Gamma \\  \mathbb  O & S_\infty\end{pmatrix}$$
is polynomially bounded  and is not similar  
to a contraction [Pi]. Since $\mu_{S_\infty}=\infty$, 
$\mu_T=\infty$, where $\mu_T$ is the multiplicity of the operator $T$. 

On the other hand, if $1\leq n <\infty$ or $1\leq k<\infty$, 
and $\Gamma$ is a Hankel operator, then 
the Foguel--Hankel operator 
$\begin{pmatrix} S_k^\ast & \Gamma \\  \mathbb  O & S_n\end{pmatrix}$
 is polynomially bounded 
 if and only if it is similar  
to a contraction [DP, Theorem 4.4 and Corollary 4.5], [Pe, Theorem 15.1.2], 
see references in [DP] and in [Pe] to the history of the question.
Therefore, if a Foguel--Hankel operator $T$
is polynomially bounded, and if $\mathcal M$ is an invariant subspace of $T$ such that 
$\mu_{T|_{\mathcal M}}<\infty$, then $T|_{\mathcal M}$ is similar to a contraction,
even if $T$ is not similar to a contraction, see Lemma \ref{lem13} below for the proof.

The purpose of this paper is to construct a polynomially bounded 
operator $T$ such that 
$T$ is not similar to a contraction and $\mu_T=1$, that is, $T$ is {\it cyclic}. 
The main result of the paper is Theorem \ref{thm71}, where a polynomially bounded 
operator $T$ 
is constructed such that 
$T$ is not similar to a contraction and there exist quasiaffinities $X$ and $Y$, 
an {\it outer} function $g\in H^\infty$, and a cyclic $C_0$-contraction $T_0$
such that $XT=T_0X$, $YT_0=TY$, and $XY=g(T_0)$. This result allows us to 
 construct  polynomially bounded 
operators  that are not similar to  contractions and 
are quasisimilar to cyclic absolutely continuous isometries: to the 
unilateral shift, to the bilateral shift, to 
absolutely continuous cyclic reductive unitary operators (Section 2).
In particular, we provide a negative answer to Question 18 from [K]: 
``if the  polynomially bounded 
operator $T$ is quasisimilar to a unitary operator, is 
$T$ similar to a contraction?". (Recall that if the polynomially 
bounded operator 
$T$ is a quasiaffine transform of a {\it singular} unitary operator, 
then $T$ is similar to this unitary operator, 
see [M], [AT], or [K, Theorem 17].)

The main construction is based on a perturbation of the sequence
 $\{R_N\}_N$ of
finite dimensional operator $R_N$ such that 
\begin{equation*}\sup_N M_{pb}(R_N)<\infty  \ \text{ and }\   
 \sup_N M_{cpb}(R_N)=\infty,\end{equation*}
  from [DP] and [Pi], see \eqref{58} below.

\smallskip

We need the following notation and definition.

$\mathcal A(\mathbb  D)$ is the {\it disk algebra}, 
and $H^\infty$ is the Banach algebra 
of all analytic bounded functions in $\mathbb  D$. 
The uniform norm on $\mathbb  D$
is denoted by the symbol $\|\cdot\|_\infty$. 

For $w\in\mathbb  D$ put 
 \begin{equation} \label{13} \beta_w(z) = \frac{w-z}{1-\overline w z}, \ \ \ z\in\mathbb  D. \end{equation} 
Clearly, $\beta_w\circ\beta_w(z) = z$ for every $z\in\mathbb  D$,
and for every $\varphi\in H^\infty$ 
we have $\varphi\circ\beta_w\in H^\infty$, 
$\|\varphi\circ\beta_w\|_\infty = \|\varphi\|_\infty$, 
 and 
 if $\varphi\in \mathcal A(\mathbb  D)$, 
then $\varphi\circ\beta_w\in \mathcal A(\mathbb  D)$.

Since polynomials are dense in $\mathcal A(\mathbb  D)$ in the uniform norm,
polynomials in \eqref{11} and \eqref{12} can be replaced by functions 
from  $\mathcal A(\mathbb  D)$.
Let $\sigma(T)\subset\overline{\mathbb  D}$, and let $w\in\mathbb  D$. Then 
the operator $\beta_w(T) = (w-T)(I-\overline w T)^{-1}$
is well defined, $p(\beta_w(T)) = (p\circ\beta_w)(T)$, 
and it is easy to see that
 \begin{equation} \label{14} M_{pb}(\beta_w(T)) = M_{pb}(T) \text{ \ and \ } 
M_{cpb}(\beta_w(T)) = M_{cpb}(T).  \end{equation} 

\smallskip

The following proposition is simple, but important in our construction.

\begin{proposition}\label{prop11}  Let $\{T_N\}_N$ be a family of operators such
that $\displaystyle{\sup_N\|T_N\|<\infty}$. Put $T = \oplus_N T_N$. Then 
$\displaystyle{M_{pb}(T)=\sup_N M_{pb}(T_N)}$ and  
$\displaystyle{M_{cpb}(T)=\sup_N M_{cpb}(T_N)}$.\end{proposition}

\begin{proof} Since $\|T\|=\sup_N\|T_N\|$ and $p(T) = \oplus_N p(T_N)$ for
every polynomial $p$, we have that $\|p(T)\| = \sup_N \|p(T_N)\|$.
Therefore, the equality for $M_{pb}$ is fulfilled.
Denote by $\mathcal H_N$ the spaces on which $T_N$ act.
Let $[p_{ij}]_{i,j=1}^n$ be a family of polynomials.
The operator $R =[p_{ij}(T)]_{i,j=1}^n$ acts on the space 
$\mathcal K = \oplus_{j=1}^n \oplus_N\mathcal H_N$. Rewriting this space as
$\mathcal K = \oplus_N\oplus_{j=1}^n \mathcal H_N$ we obtain that
$R = \oplus_N [p_{ij}(T_N)]_{i,j=1}^n$. Therefore, 
$\| [p_{ij}(T)]_{i,j=1}^n\| = \sup_N \| [p_{ij}(T_N)]_{i,j=1}^n\|$,
and we conclude that the equality for $M_{cpb}$ is fulfilled. \end{proof}

\begin{corollary}\label{cor12} Let $\{T_N\}_N$ be a family of operators such
that $\sigma(T_N)\subset\overline{\mathbb  D}$ for every index $N$, and let 
$\{w_N\}_N\subset\mathbb  D$. Suppose that $\sup_N\|\beta_{w_N}(T_N)\|<\infty$,
where $\beta_{w_N}$ are defined in \eqref{13}. Put $T=\oplus_N \beta_{w_N}(T_N)$.
Then $M_{pb}(T)=\sup_N M_{pb}(T_N)$ and  $M_{cpb}(T)=\sup_N M_{cpb}(T_N)$.
\end{corollary}

\begin{proof} The corollary follows from Proposition \ref{prop11} and the equalities 
\eqref{14}. \end{proof}

The paper is organized as follows. In Section 2, it is shown how to construct 
operators quasisimilar to isometries and having given compressions on their 
semi-invariant subspaces (not for arbitrary given compressions). 
Sections 3 and 4 contain preliminary results concerning functions and operators, 
respectively. Most of them, excepting maybe Theorem \ref{thm38}, are simple or known, 
and are given to achieve a complete exposition.
In Section 5 the sequence of finite dimensional operators from [DP], [Pi] is
described, because the knowledge of the structure of these operators is needed to construct 
an appropriate perturbation of them. In Section 6 the perturbation of 
the sequence of finite dimensional operators described in Section 5 is constructed. 
Then this perturbation is used to construct a cyclic polynomially 
bounded operator 
which is not similar to a contraction (Corollary \ref{cor64}). This operator is 
quasisimilar 
to a $C_0$-contraction, but the author doesn't know if the product of 
intertwining quasiaffinities
is an outer function of this $C_0$-contraction. In Section 7 an additional 
construction is given that guarantees the existence of such a function. 
It allows to apply the results of Section 2 to a 
$C_0$-operator constructed in Section 7.

In the rest part of Introduction, the needed notation and definitions are given, and 
Lemma \ref{lem13} 
is formulated and proved.  

Let $\mathcal H$ be a (complex, separable) Hilbert space, and let $\mathcal M$ be a
(linear, closed)  subspace.
By $I_{\mathcal H}$ and $P_{\mathcal M}$ the identical operator on $\mathcal H$ and 
the orthogonal projection from $\mathcal H$ onto $\mathcal M$ are denoted, respectively.  

Let $T$ and $R$ be operators on Hilbert spaces $\mathcal H$ and $\mathcal K$, 
respectively, and let
$X:\mathcal H\to\mathcal K$ be a linear bounded transformation such that 
$X$ {\it intertwines} $T$ and $R$,
that is, $XT=RX$. If $X$ is unitary, then $T$ and $R$ 
are called {\it unitarily equivalent}, in notation:
$T\cong R$. If $X$ is invertible, that is, $X$ has a {\it bounded} 
inverse $X^{-1}$, 
then $T$ and $R$ 
are called {\it similar}, in notation: $T\approx R$.
If $X$ is a {\it quasiaffinity}, that is, $\ker X=\{0\}$ and 
$\operatorname{clos}X\mathcal H=\mathcal K$, then
$T$ is called a {\it quasiaffine transform} of $R$, 
in notation: $T\prec R$. If $T\prec R$ and 
$R\prec T$,
 then $T$ and $R$ are called {\it quasisimilar}, 
in notation: $T\sim R$. 

The {\it multiplicity} $\mu_T$ of an operator $T$ acting 
on a space $\mathcal H$ is the minimum dimension
of its reproducing subspaces: 
$$  \mu_T=\min\{\dim E: E\subset \mathcal H, \ \ 
\bigvee_{n=0}^\infty T^n E=\mathcal H \} .$$
An operator $T$ is called {\it cyclic}, if $\mu_T=1$. 
Let  $\mathcal M$ be an invariant subspace of $T$, that is, $\mathcal M$
is a subspace of $\mathcal H$ such that $T\mathcal M\subset\mathcal M$.
It is well known and easy to see that 
 \begin{equation} \label{15}  \mu_{P_{\mathcal M^\perp}T|_{\mathcal M^\perp}}\leq\mu_T\leq
  \mu_{T|_{\mathcal M}} + \mu_{P_{\mathcal M^\perp}T|_{\mathcal M^\perp}}  \end{equation} 
and  \begin{equation} \label{16} \mu_T\geq \dim\ker T^\ast \end{equation} 
(see, for example, [Ni, II.D.2.3.1]).
It  follows immediately from the definition of quasiaffine transform that if $T\prec R$,
then $\mu_R\leq \mu_T$.

$\mathbb  D$ is the open unit disk, 
$\mathbb  T$ is the unit circle, $H^2$ is 
the Hardy space on $\mathbb  D$, $H^\infty$ is the Banach algebra of all bounded 
analytic functions in $\mathbb  D$.
The unilateral shift $S$ of multiplicity $1$ is the operator  
 of multiplication by the independent variable on $H^2$.
For an inner function  $\theta\in H^\infty$ the subspace $\theta H^2$
is invariant for $S$, put 
 \begin{equation} \label{17} \mathcal K_\theta=H^2\ominus\theta H^2 \ \ \text{ and } \ \ 
T_\theta = P_{\mathcal K_\theta}S|_{\mathcal K_\theta}.  \end{equation} 

If $T$ is a polynomially bounded operator, then $T=T_a\dotplus T_s$,
where $T_a$ is an absolutely continuous polynomially bounded operator, 
that is, a $H^\infty$-functional calculus is well-defined for $T_a$, and 
$T_s$ is similar to a singular unitary operator, see [M] or [K]. 
In this paper, absolutely continuous polynomially bounded operators 
are regarded.
An absolutely continuous polynomially bounded operator $T$ is called
a {\it $C_0$-operator}, if there exists $\varphi\in H^\infty$ such that 
$\varphi(T)=\mathbb  O$, see [BP]; if a $C_0$-operator is a contraction, it
is called a {\it $C_0$-contraction}, see [SFBK]. For an inner function $\theta$, 
$T_\theta$ is a $C_0$-contraction, because $\theta(T_\theta)=\mathbb  O$.   

For a cardinal number $n$, $0\leq n\leq\infty$, $H^2_n$ 
is the  orthogonal sum of $n$ copies of  $H^2$, and the
 unilateral shift $S_n$ is the operator 
 of multiplication by the independent variable on $H^2_n$.
It is well known and easy to see that $S_n\cong\oplus_{l=1}^n S$ and 
$\mu_{S_n}=n$.

\begin{lemma}\label{lem13} Let $1\leq n,k\leq\infty$ be cardinal numbers, and 
let $\Gamma\colon H^2_n\to H^2_k$ be an operator such that 
$S_k^\ast\Gamma = \Gamma S_n$ (that is, $\Gamma$ is a Hankel operator).
Put $$T=\begin{pmatrix} S_k^\ast & \Gamma \\  \mathbb  O & S_n\end{pmatrix},$$
 $T$ acts on $ H^2_k\oplus H^2_n$.
Suppose $T$ is polynomially bounded, and $\mathcal M$ is an invariant subspace
for $T$ such that $\mu_{T|_{\mathcal M}}<\infty$. Then $T|_{\mathcal M}$ is similar to
a contraction.\end{lemma}

\begin{proof} Put $\mathcal M_1 = (H^2_k\oplus\{0\}) \vee \mathcal M$ and 
$T_1=T|_{\mathcal M_1}$. It follows from the definition of the multiplicity that
 $\mu_{T_1}\leq \mu_{T|_{H^2_k\oplus\{0\}}} + \mu_{T|_{\mathcal M}}$.
Since $T|_{H^2_k\oplus\{0\}}=S_k^\ast$ and $\mu_{S_k^\ast}=1$
for $1\leq k\leq\infty$ (see, for example, [Ni, Lemma II.D.2.4.20]),
we obtain that $\mu_{T_1}<\infty$. 
Put $\mathcal N =  \mathcal M_1\ominus(H^2_k\oplus\{0\})$.
Then $\mathcal N\subset \{0\}\oplus H^2_n$, $S_n\mathcal N\subset\mathcal N$, and
$P_{\mathcal N}T_1|_{\mathcal N} =  S_n|_{\mathcal N}$. 
Put $\Gamma_1=\Gamma|_{\mathcal N}$.
Then  $$T_1=\begin{pmatrix} S_k^\ast & \Gamma_1 \\  \mathbb  O & S_n|_{\mathcal N}\end{pmatrix}$$
and $S_k^\ast\Gamma_1 = \Gamma_1 S_n|_{\mathcal N}$.

By \eqref{15}, $\mu_{P_{\mathcal N}T_1|_{\mathcal N}}\leq \mu_{T_1}<\infty$. 
Therefore, there exists $n_1$, $1\leq n_1<\infty$, such that 
$S_n|_{\mathcal N}\cong S_{n_1}$. Thus, 
$$T_1 \cong \begin{pmatrix} S_k^\ast & \Gamma_2 \\  \mathbb  O & S_{n_1}\end{pmatrix},$$
where $\Gamma_2$ is such that $S_k^\ast\Gamma_2 = \Gamma_2 S_{n_1}$.
Since $n_1<\infty$, $T_1$ is similar to a contraction by [DP, Theorem 4.4]. 
Since $T|_{\mathcal M} =T_1|_{\mathcal M}$, $T|_{\mathcal M}$ is similar to a contraction, too. \end{proof}

\section{Construction of operators quasisimilar to isometries}

The following definition and notation  will be used in this section.
An operator $T$ is called {\it power bounded}, 
if $\sup_{n\geq 1}\|T^n\|<\infty$.
 For an operator $T$, let $\{T\}'$ denote the commutant of $T$,
that is, the algebra of all operators that commute with $T$.

The following lemma is very simple, therefore, its proof is omitted.

\begin{lemma}\label{lem21} Suppose $\mathcal H_0$, $\mathcal K_0$, $\mathcal H_1$, $\mathcal K_1$
are Hilbert spaces, $Y_0\colon \mathcal H_0\to \mathcal K_0$, 
$Y_1\colon \mathcal H_1\to \mathcal K_1$, $Z\colon\mathcal H_1\to \mathcal K_0$ 
are operators. Put 
$$ Y = \begin{pmatrix}  Y_0 & Z \\  \mathbb  O & Y_1 \end{pmatrix}.$$
Then

\noindent
$\text{\rm (i)}$ if $Y_0$ and $Y_1$ are quasiaffinities, then $Y$ is a quasiaffinity;

\noindent
$\text{\rm(ii)}$ if $Y$ is a quasiaffinity, then 
$\operatorname{clos}Y_1\mathcal H_1 = \mathcal K_1$;

\noindent
$\text{\rm (iii)}$ if $Y_0$ is invertible and $Y$ is a quasiaffinity, then 
$Y_1$ is a quasiaffinity.\end{lemma}

  The following proposition will be applied to construct an example of
a polynomially bounded operator, which is not similar to a contraction 
and is a quasiaffine transform of $S$. 

\begin{proposition}\label{prop22}   Suppose $\mathcal H$, $\mathcal K$, $\mathcal M$
are Hilbert spaces, $T_0\colon\mathcal H\to\mathcal H$, $R_0\colon\mathcal K\to\mathcal K$,
$V\colon\mathcal M\to\mathcal M$, $Y_0\colon\mathcal H\to\mathcal K$, 
$K\colon\mathcal K\to\mathcal M$,
$Z\colon\mathcal H\to\mathcal M$ are operators. Moreover, suppose 
$Y_0T_0 = R_0Y_0$. Put $A = VZ - ZT_0 +KY_0$,
$$T = \begin{pmatrix} V & A \\  \mathbb  O & T_0 \end{pmatrix} , \ \ \ 
R = \begin{pmatrix} V & K \\  \mathbb  O & R_0 \end{pmatrix}, \ \ \text{and} \ \ \ 
Y = \begin{pmatrix} I_{\mathcal M} & Z \\  \mathbb  O & Y_0 \end{pmatrix} .$$
Then $YT = RY$. If $R$ and $T_0$ are power bounded, then $T$ is 
power bounded. If $R$ and $T_0$ are polynomially bounded, then $T$ is 
polynomially bounded. \end{proposition}

\begin{proof} The equality $YT = RY$ is a straightforward consequence of 
the definition of $T$, $R$, and $Y$ and the equality $Y_0T_0 = R_0Y_0$. 
For a polynomial $p$ set $A_{(p)} = P_{\mathcal M}p(T)|_{\mathcal H}$ and 
$K_{(p)} = P_{\mathcal M}p(R)|_{\mathcal K}$. From the equality $Yp(T) = p(R)Y$
writing in the matrix form it is easy to see that 
$$ p(T)= \begin{pmatrix} p(V) & A_{(p)} \\  \mathbb  O & p(T_0) \end{pmatrix} , 
\ \ \text{ and } \ \ A_{(p)} = p(V)Z - Zp(T_0) + K_{(p)}Y_0.$$

Now suppose that $R$ and $T_0$ are polynomially bounded. 
Since $p(V) = p(R)|_{\mathcal M}$, we have 
$\|p(V)\| \leq M_{pb}(R)\|p\|_\infty$, 
and the estimate $\|K_{(p)}\|\leq M_{pb}(R)\|p\|_\infty$ 
follows from the definition of $K_{(p)}$.
Also, $\|p(T_0)\|\leq M_{pb}(T_0)\|p\|_\infty$.
Therefore, $$\|A_{(p)}\|\leq M_{pb}(R)\|p\|_\infty \|Z\| + 
\|Z\|M_{pb}(T_0)\|p\|_\infty + M_{pb}(R)\|p\|_\infty\|Y_0\|$$ $$ = 
\bigl((M_{pb}(R) + M_{pb}(T_0))\|Z\| + M_{pb}(R)\|Y_0\|\bigr)\|p\|_\infty.$$
Since $\|p(T)\|\leq\sqrt 3 \max(\|p(V)\|, \|A_{(p)}\|, \|p(T_0)\|)$,
we obtain that 
$$\|p(T)\|\leq\sqrt 3 \max\bigl(M_{pb}(R), 
(M_{pb}(R) + M_{pb}(T_0))\|Z\| + M_{pb}(R)\|Y_0\|, M_{pb}(T_0)\bigr)\|p\|_\infty.$$
Thus, if $R$ and $T_0$ are polynomially bounded, then $T$ is 
polynomially bounded.

If we suppose only that $R$ and $T_0$ are power bounded, then 
the proof of the power boundedness of $T$ is the same. \end{proof}

\begin{corollary}\label{cor23}   Suppose $\theta\in H^\infty$ is an inner function, 
$\mathcal K_\theta = H^2\ominus\theta H^2$, 
$T_\theta = P_{\mathcal K_\theta} S|_{\mathcal K_\theta}$,
$K_\theta = P_{\theta H^2} S|_{\mathcal K_\theta}$,
$\mathcal H$ is a Hilbert space, $T_0\colon\mathcal H\to\mathcal H$, 
$Y_0\colon\mathcal H\to\mathcal K_\theta$, 
$Z\colon\mathcal H\to\theta H^2$ are operators. 
Moreover, suppose 
$Y_0T_0 = T_\theta Y_0$. Put $A = S|_{\theta H^2}Z - ZT_0 +K_\theta Y_0$,
$$T = \begin{pmatrix} S|_{\theta H^2} & A \\  \mathbb  O & T_0 \end{pmatrix}  
\ \ \ \text{and} \ \ \ 
Y = \begin{pmatrix} I_{\theta H^2} & Z \\  \mathbb  O & Y_0 \end{pmatrix} .$$
Then $YT = SY$. If $Y_0$ is a quasiaffinity, then 
$Y$ is a quasiaffinity. If $T_0$ is polynomially bounded, then $T$ is 
polynomially bounded. If $T_0$ is not similar to a contraction, 
then $T$ is not similar to a contraction.\end{corollary}

\begin{proof} Put $\mathcal M = \theta H^2$, $V = S|_{\theta H^2}$, 
$\mathcal K = \mathcal K_\theta$, $K = K_\theta$, $R_0=T_\theta$,
and apply Proposition \ref{prop22}. Then $R=S$, therefore, $YT = SY$.
By Proposition \ref{prop22}, the polynomially boundedness of $T_0$
implies the polynomially boundedness of  $T$. 
If $Y_0$ is a quasiaffinity,  then  $Y$ is a quasiaffinity
by Lemma \ref{lem21} (i).
If $T$ is similar to a contraction, the same holds for 
the compression of $T$ to any  semi-invariant subspace,
in particular, for $T_0$. Therefore, if  $T_0$ is not similar to a contraction, 
then $T$ is not similar to a contraction. \end{proof}

\begin{proposition}\label{prop24} Suppose $\theta\in H^\infty$ is an inner function, 
$T$ is the operator from Corollary \ref{cor23},
and $Y_0$ from Corollary  \ref{cor23} is a quasiaffinity. Then $T\sim S$ if and only if there exist  
a quasiaffinity $X_0\colon\mathcal K_\theta\to\mathcal H$ and an outer function
$g\in H^\infty$ such that $T_0X_0 = X_0T_\theta$ and $Y_0X_0 = g(T_\theta)$.\end{proposition}

\begin{proof} {\it The ``if" part.} Put 
$K_{(g)} = P_{\theta H^2}g(S)|_{\mathcal K_\theta}$, 
$W = K_{(g)} - ZX_0$, and 
$$ X = \begin{pmatrix} g(S)|_{\theta H^2} & W \\  \mathbb  O & X_0 \end{pmatrix} .$$ 
Since $g$ is outer, $g(S)|_{\theta H^2}$ is a quasiaffinity, and  
$X$ is a quasiaffinity by Lemma  \ref{lem21}  (i). 
It remains to prove that $XS = TX$.
Writing this equality in matrix form, it is easy to see
that it is sufficient to prove the equality
 \begin{equation} \label{21} g(S)|_{\theta H^2}K_\theta + W T_\theta = S|_{\theta H^2}W + A X_0. \end{equation} 
Using the definition of $W$ and $A$ we obtain that 
$$g(S)|_{\theta H^2}K_\theta + W T_\theta = 
g(S)|_{\theta H^2}K_\theta + K_{(g)}T_\theta - ZX_0T_\theta,$$
and $$S|_{\theta H^2}W + A X_0 = 
S|_{\theta H^2}K_{(g)} - S|_{\theta H^2}ZX_0 +
S|_{\theta H^2}ZX_0 - ZT_0X_0 + K_\theta Y_0X_0$$
$$ = 
S|_{\theta H^2}K_{(g)} - ZT_0X_0 + K_\theta g(T_\theta).$$
Using the equality $T_0X_0 = X_0T_\theta$, we infer that \eqref{21}
is equivalent to the equality 
$$g(S)|_{\theta H^2}K_\theta + K_{(g)}T_\theta = 
S|_{\theta H^2}K_{(g)} + K_\theta g(T_\theta).$$
But this equality follows from the equality $g(S)S =Sg(S)$ 
written in the matrix form.

{\it The ``only if" part.} Suppose that 
a quasiaffinity $Y$ is from Corollary \ref{cor23}, and
$X\colon H^2 \to \theta H^2\oplus \mathcal H$
is a quasiaffinity such that $XS = TX$.
Since $YX\in\{S\}'$, there exists a function $g\in H^\infty$ 
such that  $YX = g(S)$, and, since $YX$ is a quasiaffinity,
$g$ is outer. Writing 
$X = \begin{pmatrix} X_1 & W \\  X_2  & X_0 \end{pmatrix}$  
with respect to the decompositions of the spaces
$H^2 = \theta H^2 \oplus \mathcal K_\theta$ and 
$\theta H^2\oplus \mathcal H$ and regarding the equality
$YX = g(S)$ with respect to these decompositions,
 we obtain that $Y_0X_0 = g(T_\theta)$ and
$Y_0X_2 = \mathbb  O$. Since $\ker Y_0=\{0\}$, we conclude that
$X_2 = \mathbb  O$. 
Thus, $X = \begin{pmatrix} X_1 & W \\  \mathbb  O  & X_0 \end{pmatrix}$.
By Lemma  \ref{lem21}  (ii), $\operatorname{clos}X_0\mathcal K_\theta =\mathcal H$.
Since $\ker g(T_\theta)=\{0\}$, from the equality 
$Y_0X_0 = g(T_\theta)$ we conclude that $\ker X_0=\{0\}$.
From the equality $XS=TX$ we conclude that 
$X_0 T_\theta = T_0X_0$. \end{proof}

\begin{remark}\label{rem25} The conditions $T_0X_0 = X_0T_\theta$ and 
$Y_0T_0 = T_\theta Y_0$ guarantee that $Y_0X_0 \in\{T_\theta\}'$, 
consequently, there exists a function $\varphi\in H^\infty$ such that
 $Y_0X_0 = \varphi(T_\theta)$. The condition that 
 $X_0$ and  $Y_0$ are quasiaffinities guarantees that 
the inner factor of $\varphi$ is coprime with $\theta$. 
For every $f\in H^\infty$ such that the inner factor of $f$ is coprime with $\theta$
one can regard $f(T_\theta) Y_0$ instead of $Y_0$, because $f(T_\theta) Y_0$ 
is a quasiaffinity which intertwines $T_\theta$ with $T_0$. 
Also, for every $h\in H^\infty$ the equality 
$(f\varphi)(T_\theta) = (\varphi f  + \theta h)(T_\theta)$ holds. 
But there exist functions  $\varphi$ and $\theta$ 
such that $\varphi$ is coprime with $\theta$ and the function
$\varphi f + \theta h$ is not outer for every $f$, $h\in H^\infty$,
see [No] for (it seems the first) example of such functions 
$\varphi$ and $\theta$.
\end{remark}

Now we see that to construct a polynomially bounded operator
$T$ such that $T\prec S$ and $T$ is not similar to a contraction, 
it is sufficient to construct  a polynomially bounded operator
$T_0$ such that $T_0 \sim T_\theta$ for some inner function 
$\theta$, and $T_0$ is not similar to a contraction. 
This is done in Corollary \ref{cor64}.
 To construct a polynomially bounded operator
$T$ such that $T\sim S$ and $T$ is not similar to a contraction, 
it is sufficient to construct  a polynomially bounded operator
$T_0$ such that $T_0 \sim T_\theta$ for some inner function 
$\theta$, $T_0$ is not similar to a contraction, and 
there exist  quasiaffinities $X_0$ and $Y_0$ and 
an outer function
$g\in H^\infty$ such that $T_0X_0 = X_0T_\theta$, 
$Y_0T_0 = T_\theta Y_0$,
 and $Y_0X_0 = g(T_\theta)$. This is done in Theorem \ref{thm71}. The function 
$g$ is from Lemma \ref{lem34}.
The operator $Z$ from Proposition \ref{prop22} can be zero in both cases.
The operator $Z$ was considered to show that the choice of $Z$ does not allow to 
avoid the condition on the existence of an {\it outer} function $g$
such that $Y_0X_0 = g(T_\theta)$ in this construction.

\smallskip 

Put $\chi(z) = z$, $z\in\mathbb  T$. Let $U$ be the bilateral shift of 
multiplicity 1, that is, the operator of multiplication by $\chi$ 
on $L^2(\mathbb  T)$. $U$ has the following form with respect to 
the decomposition $L^2(\mathbb  T) = H^2 \oplus H^2_-$:
 \begin{equation} \label{22} U = \begin{pmatrix} S & K \\  \mathbb  O & S_\ast \end{pmatrix}, \end{equation}
where $S$  is the unilateral shift of multiplicity 1,
$K\colon H^2_- \to H^2$ acts by the formula $K\chi^{-n}=0$, $n\geq 2$,
$K\chi^{-1}=\chi^0$, $S_\ast\colon H^2_- \to H^2_-$ 
acts by the formula $S_\ast\chi^{-n}=\chi^{-n+1}$, $n\geq 2$,
$S_\ast\chi^{-1}=0$. 

 The following propositions will be applied to construct an example of
a polynomially bounded operator, which is not similar to a contraction 
and is quasisimilar to $U$. 

\begin{proposition}\label{prop26}   Suppose $\mathcal H$ is a
Hilbert space, and $T_1\colon\mathcal H\to\mathcal H$ and
$Y_1\colon H^2 \to \mathcal H$ are operators such that 
$Y_1S = T_1Y_1$. Put $x_0 = Y_1\chi^0$, 
$A\colon H^2_-\to \mathcal H$, $A\chi^{-1}=x_0$, $A\chi^{-n}=0$, $n\geq 2$,
$$T = \begin{pmatrix} T_1 & A \\  \mathbb  O & S_\ast \end{pmatrix} 
\ \ \text{and} \ \ \ 
Y = \begin{pmatrix} Y_1 & \mathbb  O \\  \mathbb  O & I_{H^2_-} \end{pmatrix} .$$
Then $YU = TY$.

 If $T_1$ is power bounded, then $T$ is 
power bounded. If  $T_1$ is polynomially bounded, then $T$ is 
polynomially bounded. \end{proposition}

\begin{proof} Clearly, $Y_1K=A$. From this equality and \eqref{22} 
we conclude that  $YU = TY$. Note that  
$$T^\ast = \begin{pmatrix} (S_\ast)^\ast & A^\ast \\  \mathbb  O & 
T_1^\ast \end{pmatrix} $$
with respect to the decomposition $H^2_-\oplus\mathcal H$.
Applying Proposition \ref{prop22} with
$T_0=T_1^\ast$, $R_0=S^\ast$, $V=(S_\ast)^\ast$, $Y_0=Y_1^\ast$, 
$Z=\mathbb  O$, $K^\ast$ instead of $K$ and $A^\ast$ instead of $A$, and 
taking into account that $R=U^\ast$, we obtain the conclusion 
of Proposition \ref{prop22}  for $T^\ast$, and, consequently, for $T$.
 \end{proof}

\begin{proposition}\label{prop27}  Suppose 
$g\in H^\infty$ is an outer function, $\mathcal H$ is a
Hilbert space, $T_1\colon\mathcal H\to\mathcal H$ is an operator,
$Y_1\colon H^2 \to \mathcal H$ and $X_1\colon \mathcal H \to H^2$
are quasiaffinities such that 
$Y_1S = T_1Y_1$, $X_1T_1 = SX_1$, and $X_1Y_1 =g(S)$.
Let $T$ and $Y$ be defined as in  Proposition \ref{prop26}.
Put $K_{(g)} = P_{H^2}g(U)|_{H^2_-}$ and
$$X = \begin{pmatrix} X_1 & K_{(g)}\\  \mathbb  O & g(S_\ast) \end{pmatrix}.$$
Then $X$ and $Y$ are quasiaffinities such that 
$YU = TY$, $XT = UX$, and $XY =g(U)$.\end{proposition}

\begin{proof} Since $g$ is outer, $g(S_\ast)$ is a quasiaffinity.
By Lemma \ref{lem21} (i), $X$ and $Y$ are quasiaffinities.
The equality $YU = TY$ is proved in Proposition \ref{prop26}. 
Using the matrix forms of $U$, $T$, and $X$, and the equality 
$X_1T_1 = SX_1$, it is easy to see that the equality $XT = UX$
follows from the equality
 \begin{equation} \label{23} X_1A + K_{(g)} S_\ast = S K_{(g)} + K g(S_\ast). \end{equation} 
We show that $X_1A = g(S)K$. Indeed, for $n\geq 2$ we have that
$A\chi^{-n} = 0$ and $K\chi^{-n} = 0$. Furthermore, 
$X_1A\chi^{-1} = X_1x_0 = X_1 Y_1\chi^0 = g$, and 
$g(S)K\chi^{-1} = g(S)\chi^0 = g$.
To prove \eqref{23} it remains to note that 
$S K_{(g)} + K g(S_\ast) = g(S)K + K_{(g)} S_\ast$, and
this equality  follows from the equality $g(U)U =Ug(U)$ 
written in the matrix form.
The equality $XY =g(U)$ can be easily obtained from
the definitions of $X$ and $Y$ and the matrix form of $g(U)$. \end{proof}

\begin{corollary}\label{cor28}  Suppose $T_1$ satisfies the conditions of
Proposition \ref{prop27}, and $T$ is defined in Proposition \ref{prop27}.
If $T_1$ is polynomially bounded, then $T$ is 
polynomially bounded. If $T_1$ is not similar to a contraction, 
then $T$ is not similar to a contraction.\end{corollary} 

\begin{proof} The assertion about polynomial boundedness is proved
in Proposition \ref{prop26}. The assertion about similarity to a contraction
follows from the same reasons as at the end of the proof of 
Corollary \ref{cor23}. \end{proof}

The following lemma shows that there exist polynomially 
bounded operators that are quasisimilar to cyclic 
reductive unitaries and are not similar to contractions.
A function $g$ in the example will be
$g(z)=\exp(-(\frac{1+z}{1-z})^\alpha)$, 
$z\in\operatorname{clos}\mathbb  D\setminus\{1\}$, $0<\alpha<1$, 
see Lemma \ref{lem34}.
The interested readers
 can find sets $\sigma$ 
satisfying the condition of the lemma themselves.

 \begin{lemma}\label{lem29} Suppose  $T$ is an operator, and 
$X$ and $Y$ are quasiaffinities such that 
$YU = TY$, and $XT = UX$.
Since $XY\in\{U\}'$, there exists a function $g\in L^\infty$ 
such that  $XY =g(U)$.
Let $\sigma\subset\mathbb  T$ be a measurable set  
and let $\delta>0$ be such that 
$|g|\geq \delta$ a.e. on $\sigma$.
Put $\mathcal M = \operatorname{clos}YL^2(\sigma)$,
where $L^2(\sigma)=\{h\in L^2: h=0$  a.e. on $\mathbb  T\setminus\sigma\}$.
Then  \begin{equation} \label{24} T|_{\mathcal M} \approx U|_{L^2(\sigma)}   \end{equation} 
and $P_{\mathcal M^\perp}T|_{\mathcal M^\perp}\sim U|_{L^2(\mathbb  T\setminus\sigma)}$.
If $T$ is power bounded and is not similar to a contraction, then 
$P_{\mathcal M^\perp}T|_{\mathcal M^\perp}$ is not similar to a contraction. \end{lemma}

\begin{proof} Clearly, $T \mathcal M \subset \mathcal M$, and 
$$\operatorname{clos}X\mathcal M = \operatorname{clos}XYL^2(\sigma)=
\operatorname{clos}g(U)L^2(\sigma)=L^2(\sigma).$$
Put $$X_\sigma =  X|_{\mathcal M}\colon\mathcal M\to L^2(\sigma) \  \ \text{ and } \ \ 
Y_\sigma =  Y|_{L^2(\sigma)}\colon L^2(\sigma)\to\mathcal M.$$
Then $X_\sigma$ and  $Y_\sigma$ are quasiaffinities. 
Since $X_\sigma Y_\sigma = g(U)|_{L^2(\sigma)}$ is invertible, 
we conclude that 
$X_\sigma$ and  $Y_\sigma$ are invertible. Since 
$X_\sigma T|_{\mathcal M} = U|_{L^2(\sigma)}X_\sigma $,
\eqref{24} is proved.

Put $T_1 = P_{\mathcal M^\perp}T|_{\mathcal M^\perp}$, 
$X_1 = P_{L^2(\mathbb  T\setminus\sigma)} X|_{\mathcal M^\perp}$, and 
$Y_1 = P_{\mathcal M^\perp}Y|_{L^2(\mathbb  T\setminus\sigma)}$. 
It is easy to see that $X_1T_1 = U|_{L^2(\mathbb  T\setminus\sigma)}X_1$ and 
$Y_1U|_{L^2(\mathbb  T\setminus\sigma)} = T_1Y_1$. 
By Lemma \ref{lem21} (iii), $X_1$ and $Y_1$ are quasiaffinities. 
By [B, Corollary 2.2] applied to $T^\ast$ we have that
$T\approx  U|_{L^2(\sigma)}\oplus T_1$. 
If $T_1$ is similar to a contraction, we conclude that 
$T$ is similar to a contraction. 
Also, to show that $T$ is similar to a contraction, 
if $T_1$ is similar to a contraction, 
[C, Corollary 4.2] can be applied. \end{proof}

\section{Preliminaries: function theory}

For $\lambda\in\mathbb  D$ denote by $b_\lambda$  a Blaschke factor:
 $b_\lambda(z)=\frac{|\lambda|}{\lambda}
\frac{\lambda-z}{1-\overline\lambda z}$, $z\in\mathbb  D$.
Note that 
 \begin{equation} \label{31} b_\lambda\circ\beta_w = \zeta_{w,\lambda}b_{\beta_w(\lambda)}, \ \
\text{ where }   \zeta_{w,\lambda}\in\mathbb  T.  \end{equation}  
Recall that $\beta_w$ is defined in \eqref{13}.

\smallskip

The following lemma is very simple, but useful.

\begin{lemma}\label{lem31} Let $\{B_N\}_N$ be a sequence of
 finite Blaschke products.
Then there exists a sequence $\{r_N\}_N$ such that $0<r_N<1$
and for every sequence $\{w_N\}_N\subset\mathbb  D$ such that 
$|w_N|\geq r_N$ there exists 
a sequence $\{\zeta_N\}_N\subset\mathbb  T$ such that 
 the product $\prod_N\zeta_NB_N\circ\beta_{w_N}$ converges. \end{lemma}

\begin{proof} Denote by $\Lambda_N$ the set of zeros of $B_N$, 
then $B_N=\prod_{\lambda\in\Lambda_N}b_{\lambda}^{k_\lambda}$, 
where $1\leq k_\lambda<\infty$ is the multiplicity of $\lambda$.
For $w_N\in\mathbb  D$ put $\zeta_N = 
\prod_{\lambda\in\Lambda_N}\overline\zeta_{w_N,\lambda}^{k_\lambda}$, 
where $\zeta_{w_N,\lambda}$ are from \eqref{31},
then $\zeta_NB_N\circ\beta_{w_N}=
\prod_{\lambda\in\Lambda_N}b_{\beta_{w_N}
(\lambda)}^{k_\lambda}.$
Let $\{\eta_N\}_N$ be a sequence such that $\eta_N>0$ and
$\sum_N\eta_N<\infty$. Since $|\beta_w(\lambda)|\to 1$ when 
$|w|\to 1$ and $\Lambda_N$ is finite, there exists $0<r_N<1$ such that 
$\sum_{\lambda\in\Lambda_N}k_\lambda(1-|\beta_w(\lambda)|)\leq \eta_N$
for $w\in\mathbb  D$, $|w|\geq r_N$.
Let $w_N\in\mathbb  D$, and let $|w_N|\geq r_N$. Then 
$$\sum_N\sum_{\lambda\in\Lambda_N}k_\lambda(1-|\beta_{w_N}(\lambda)|)
\leq \sum_N\eta_N<\infty,$$ that is, the product 
$\prod_N\prod_{\lambda\in\Lambda_N}b_{\beta_{w_N}(\lambda)}^{k_\lambda} = 
\prod_N\zeta_NB_N\circ\beta_{w_N}$ converges. \end{proof}

The following lemma is a particular case of [SFBK, VI.1.6].
Also, it can be proved in a straightforward manner.

\begin{lemma}\label{lem32} Let $w\in\mathbb  D$. Put 
$$U_w\colon H^2\to H^2,\  \  (U_wh)(z) = 
\frac{(1-|w|^2)^{1/2}}{1-\overline wz}(h\circ\beta_w)(z), \  \
z\in\mathbb  D, \ \ h\in H^2, $$
 where $\beta_w$ is defined in \eqref{13}. Then $U_w$ is unitary, 
$U_w = U_w^{-1}$, 
$U_w\mathcal K_{\theta\circ\beta_w} =\mathcal K_\theta$
and $U_w|_{\mathcal K_{\theta\circ\beta_w}} T_{\theta\circ\beta_w} = 
\beta_w(T_\theta)U_w|_{\mathcal K_{\theta\circ\beta_w}}$
 for every inner function $\theta\in H^\infty$, where the space 
 $\mathcal K_\theta$ and the operator 
 $T_\theta$ are defined in \eqref{17}.\end{lemma}

The following lemma can be easily proved by induction. Therefore, 
its proof is omitted.

\begin{lemma}\label{lem33} There exist functions 
$c_{nk}\colon\mathbb  D\to\mathbb  C$
such that $\sup_{\mathbb  D}|c_{nk}|<\infty$ for every $n\geq 1$, 
$0\leq k\leq n-1$,
and for every analytic function $g\colon\mathbb  D\to\mathbb  C$ and every 
$w\in\mathbb  D$
 \begin{equation} \label{32} (g\circ\beta_w)^{(n)}(z) = \sum_{k=0}^{n-1}
g^{(n-k)}(\beta_w(z))\frac{c_{nk}(w)}{(1-\overline wz)^{2n-k}}.  \end{equation} 
Namely, $c_{10}(w)=|w|^2-1$, $c_{n+1,0}(w)=(|w|^2-1)c_{n0}(w)$, 
$$c_{n+1,k}(w)=(|w|^2-1)c_{nk}(w)+(2n-k+1)\overline w c_{n,k-1}(w), 
\ \ \ 1\leq k\leq n-1,$$
$c_{n+1,n}(w)= (n+1)\overline w c_{n,n-1}(w)$. \end{lemma}

We need the following lemmas.

\begin{lemma}\label{lem34} There exists an outer function $g\in H^\infty$
such that $g^{(n)}(r)\to 0$ when $r\in(0,1)$, $r\to 1$, for every $n\geq 0$. \end{lemma}

\begin{proof}  Let $0<\alpha<1$, put $g(z)=\exp(-(\frac{1+z}{1-z})^\alpha)$, 
$z\in\mathbb  D$.
It is easy to prove by induction that 
$$g^{(n)}(z)=g(z)\sum_{l=1}^{\kappa(n)}a_{nl}
\Bigl(\frac{1+z}{1-z}\Bigr)^{\gamma_{nl}}(1-z)^{\eta_{nl}}, \ \ \ z\in\mathbb  D,$$
 where $\kappa(n)<\infty$, $a_{nl}$, $\gamma_{nl}$, $\eta_{nl}\in \mathbb  R$.
Therefore, $g^{(n)}(r)\to 0$ when $r\in(0,1)$, $r\to 1$, for every $n\geq 0$.
In [Ni, Example I.A.4.3.7, p.71] it is proved that $g\in H^\infty$
and $g$ is outer.\end{proof} 

\begin{lemma}\label{lem35}  Suppose  $\Lambda\subset\mathbb  D$ 
is finite, $1\leq k_\lambda<\infty$ for every $\lambda\in\Lambda$,
and $B=\prod_{\lambda\in\Lambda}b_{\lambda}^{k_\lambda}$. 
Then there exists $C>0$ which depends on $B$ such that
$$\operatorname{dist}(\varphi, B H^\infty)\leq 
C\max_{\lambda\in\Lambda, 0\leq k\leq k_\lambda-1}|\varphi^{(k)}(\lambda)|
\ \ \text{ for every } \ \varphi\in  H^\infty.$$
\end{lemma}

\begin{proof}  Recall that for the Blaschke product $B$ the space $\mathcal K_B$ 
and the operator $T_B$ on $\mathcal K_B$ are defined by \eqref{17}. 
Put $u_{\lambda k}(z)=\frac{B(z)}{(z-\lambda)^k}$, $z\in\mathbb  D$,
$\lambda\in\Lambda$, $1\leq k\leq k_\lambda$. Then 
$\{u_{\lambda k}\}_{\lambda\in\Lambda,1\leq k\leq k_\lambda}$
is a basis of $\mathcal K_B$, and, since $\dim \mathcal K_B<\infty$,
there exist an invertible operator $X$ on $\mathcal K_B$  and an orthonormal 
basis $\{e_{\lambda k}\}_{\lambda\in\Lambda,1\leq k\leq k_\lambda}$ of 
$\mathcal K_B$ such that $Xu_{\lambda k}=e_{\lambda k}$, 
$\lambda\in\Lambda$, $1\leq k\leq k_\lambda$.
Let $\varphi \in H^\infty$. Then
$$ X\varphi(T_B)X^{-1}e_{\lambda k}=
\sum_{l=0}^{k-1}\frac{\varphi^{(l)}(\lambda)}{l!}e_{\lambda, k-l}, \ \ 
\lambda\in\Lambda,  \ \ 1\leq k\leq k_\lambda,$$ therefore,
there exists $C_1>0$ which depends on $\sum_{\lambda\in\Lambda}k_\lambda$ only  
such that $\|X\varphi(T_B)X^{-1}\|\leq
C_1\max_{\lambda\in\Lambda, 0\leq k\leq k_\lambda-1}|\varphi^{(k)}(\lambda)|$.
Thus, 
$$\|\varphi(T_B)\|\leq \|X\|\|X^{-1}\|
C_1\max_{\lambda\in\Lambda, 0\leq k\leq k_\lambda-1}|\varphi^{(k)}(\lambda)|.$$
By Nehari's theorem, 
$\operatorname{dist}(\varphi,B H^\infty)=\|\varphi(T_B)\|$.\end{proof} 

\begin{lemma}\label{lem36}  Suppose $B$ is a finite Blaschke product,
and $a,\gamma\in (0,1)$. Then there exists $r\in (0,1)$
such that $|(B\circ\beta_w)(z)|\geq\gamma$ for 
$z$, $w\in\mathbb  D$, $|z|\leq a$, $r\leq |w|$, where $\beta_w$ is defined in \eqref{13}.\end{lemma}

\begin{proof}  Set 
$B=\prod_{\lambda\in\Lambda}b_\lambda^{k_\lambda}$, 
where $\Lambda\subset\mathbb  D$ is a finite set, and
$1\leq k_\lambda<\infty$ for every $\lambda\in\Lambda$.
Put $\kappa=\sum_{\lambda\in\Lambda}k_\lambda$. 

It is easy to see from \eqref{31} that 
$$1-|(b_\lambda\circ\beta_w)(z)|^2 = 
\frac{(1-|\beta_w(\lambda)|^2)(1-|z|^2)}{|1-\overline\beta_w(\lambda)z|^2}.$$
From this equality and  the relations 
$ |\beta_w(\lambda)|\to 1$ when $w\in\mathbb  D$, $|w|\to 1$,
we have that
$|(b_\lambda\circ\beta_w)(z)| \to 1$ when $w\in\mathbb  D$, $|w|\to 1$ 
uniformly on $\{z: |z|\leq a\}$ for every $\lambda\in\mathbb  D$.
Therefore, for every $\lambda\in\mathbb  D$ there exists 
$r_\lambda\in (0,1)$ such that 
$|(b_\lambda\circ\beta_w)(z)| \geq  \gamma^{1/\kappa}$ 
for $z$, $w\in\mathbb  D$, $|z|\leq a$, $r_\lambda\leq |w|$.
Put $r=\max_{\lambda\in\Lambda}r_\lambda$. Then
$$|(B\circ\beta_w)(z)| = 
\prod_{\lambda\in\Lambda}|(b_\lambda\circ\beta_w)(z)|^{k_\lambda}\geq
\prod_{\lambda\in\Lambda}(\gamma^{1/\kappa})^{k_\lambda}=\gamma$$
 for $z$, $w\in\mathbb  D$, $|z|\leq a$, $r\leq |w|$. \end{proof} 

\begin{proposition}\label{prop37}  Suppose $C>0$, $\Lambda\subset (0,1)$ 
is finite, $1\leq k_\lambda<\infty$ for every $\lambda\in\Lambda$,
$B=\prod_{\lambda\in\Lambda}b_{\lambda}^{k_\lambda}$,
$\varphi$, $g \in H^\infty$, $\varphi(\lambda)\neq 0$ 
for every $\lambda\in\Lambda$, 
and $g^{(k)}(r)\to 0$ when 
$r\in(0,1)$, $r\to 1$, for every $k\geq 0$.
Then there exists $r_0$, $0<r_0<1$, such that
for every $w\in[r_0,1)$ there exists 
$f \in H^\infty$ such that $g\circ\beta_w -f\varphi\in BH^\infty$
and $\operatorname{dist}(f,BH^\infty)\leq C$, where $\beta_w$ is defined in \eqref{13}.
\end{proposition}

\begin{proof}  We have $g\circ\beta_w -f\varphi\in BH^\infty$
if and only if 
 \begin{equation} \label{33} (g\circ\beta_w)^{(n)}(\lambda)=(f\varphi)^{(n)}(\lambda) \ \ 
\text{ for every } \ \lambda\in\Lambda, \ 0\leq n\leq k_\lambda-1.  \end{equation} 
Put $$\varphi_{\lambda n k}=\frac{n!}{k!(n-k)!}\varphi^{(n-k)}(\lambda), \ \ \
0\leq k\leq n, \ \ 0\leq k\leq  k_\lambda-1,$$
$$ \varphi_{\lambda n k}=0, \ \ \ 0\leq n\leq k-1\leq k_\lambda-1.$$
 Then \eqref{33} is equivalent to 
 \begin{equation} \label{34} \{\varphi_{\lambda n k}\}_{n,k=0}^{k_\lambda-1}\cdot
\{f^{(k)}(\lambda)\}_{k=0}^{k_\lambda-1}=
\{(g\circ\beta_w)^{(n)}(\lambda)\}_{n=0}^{k_\lambda-1}.  \end{equation} 
Since $\{\varphi_{\lambda n k}\}_{n,k=0}^{k_\lambda-1}$ 
is a lower triangular matrix, and
$\varphi_{\lambda n n}=\varphi(\lambda)$, $0\leq n\leq k_\lambda-1$,
$\det\{\varphi_{\lambda n k}\}_{n,k=0}^{k_\lambda-1}=
\varphi(\lambda)^{k_\lambda}\neq 0$ for $\lambda\in\Lambda$,
therefore, the matrix $\{\varphi_{\lambda n k}\}_{n,k=0}^{k_\lambda-1}$
is invertible. Put $\Phi_\lambda=
(\{\varphi_{\lambda n k}\}_{n,k=0}^{k_\lambda-1})^{-1}$, then \eqref{34} is equivalent to
 \begin{equation} \label{35} \{f^{(k)}(\lambda)\}_{k=0}^{k_\lambda-1}=\Phi_\lambda\cdot 
\{(g\circ\beta_w)^{(n)}(\lambda)\}_{n=0}^{k_\lambda-1}.  \end{equation} 
Since $\Lambda$ is a finite set and $k_\lambda<\infty$, 
for every $w\in\mathbb  D$ there exists a function
$f \in H^\infty$ which satisfies \eqref{35}. 
By Lemma \ref{lem35}, there exists $C_1$ which depends on $B$ 
only such that $\operatorname{dist}(f, B H^\infty)\leq 
C_1\max_{\lambda\in\Lambda, 0\leq k\leq k_\lambda-1}|f^{(k)}(\lambda)|$.
Since $\Lambda\subset (0,1)$, 
$\beta_w(\lambda)\in(-1,1)$, if $w\in(0,1)$, and 
$\beta_w(\lambda)\to 1$ when $w\in(0,1)$, $w\to 1$.
From \eqref{32} and the condition on $g$ we conclude that 
$$\|\{(g\circ\beta_w)^{(n)}(\lambda)\}_{n=0}^{k_\lambda-1}\|\to 0 \text{ \ when \ }w\to 1, \ \ w\in(0,1),$$
for every $\lambda\in\Lambda$.
Therefore, we infer from \eqref{35} that for every $\lambda\in\Lambda$
there exists $r_\lambda \in(0,1)$ such that
$\|\{f^{(k)}(\lambda)\}_{k=0}^{k_\lambda-1}\|\leq C/C_1$, 
if $w\in [r_\lambda,1)$.
Now $r_0= \max_{\lambda\in\Lambda}r_\lambda <1$
satisfies to the conclusion of the proposition. \end{proof} 

The following theorem is the main result of this section. 
The condition \eqref{36} from the theorem is 
the {\it generalized Carleson condition}, see [Ni, Theorem II.C.3.2.14, p.164].

Recall that $\beta_w$ is defined in \eqref{13}.

 \begin{theorem}\label{thm38} Suppose $C>0$, 
$g \in H^\infty$, and $g^{(k)}(r)\to 0$ when 
$r\in(0,1)$, $r\to 1$, for every $k\geq 0$.
Furthermore, suppose 
 $B_N$ are finite Blaschke products with zeros from $(0,1)$, 
$\varphi_N \in H^\infty$, $\varphi_N(\lambda)\neq 0$ 
for every $\lambda\in\mathbb  D$ such that $B_N(\lambda)=0$, 
 for every index $N$.
Then there exist $\delta>0$ and sequences of $w_N\in(0,1)$, 
of $\zeta_N\in\mathbb  T$, and 
of $\psi_N \in H^\infty$ such that 
the product $\prod_N\zeta_NB_N\circ\beta_{w_N}$ converges,
 \begin{equation} \label{36} 
\bigl|\prod_N\zeta_N(B_N\circ\beta_{w_N})(z)\bigr|\geq\delta\inf_N|(B_N\circ\beta_{w_N})(z)|
 \ \ \text{ for every } \ z\in\mathbb  D,  \end{equation} 
$g-\psi_N\cdot\varphi_N\circ\beta_{w_N}\in (B_N\circ\beta_{w_N}) H^\infty$,
 and 
$\operatorname{dist}(\psi_N, (B_N\circ\beta_{w_N}) H^\infty)\leq C$.\end{theorem}

\begin{proof}  Denote  by $\{r_{1N}\}_N$ a sequence from Lemma \ref{lem31} applied to
$\{B_N\}_N$, and by $r_{2N}$ a quantity from Proposition \ref{prop37} 
applied to $C$, $B_N$, $\varphi_N$, and $g$.
Put $r_N=\max(r_{1N},r_{2N})$. 

Let $\gamma_n\in (0,1)$ be such that $\delta =\prod_n\gamma_n$ 
converges, put $\delta_N =\prod_{n=1}^N\gamma_n$.
We construct the sequence $\{w_N\}_N$ such that $w_N\in[r_N,1)$
by induction. Let $w_1\in[r_1,1)$ be arbitrary. 
Clearly, $|(B_1\circ\beta_{w_1})(z)|\geq\delta_1|(B_1\circ\beta_{w_1})(z)|$
for every $z\in \mathbb  D$.
Suppose that $w_n\in[r_n,1)$, $1\leq n\leq N$, are such that
 \begin{equation} \label{37}|\theta_N(z)|\geq\delta_N\inf_{1\leq n\leq N}|(B_n\circ\beta_{w_n})(z)| \ \ 
\text{ for every } z\in \mathbb  D, 
\  \text{ where } \theta_N =\prod_{n=1}^N B_n\circ\beta_{w_n}. \end{equation} 
Since $\theta_N$ is a finite Blaschke product, there exists 
$a\in(0,1)$ such that 
 \begin{equation} \label{38} |\theta_N(z)|\geq\gamma_{N+1} \ \  \text{ for every } z\in \mathbb  D, \ \
 |z|\geq a.  \end{equation} 
We have from \eqref{38} that 
 \begin{equation} \label{39} \inf_{1\leq n\leq N}|(B_n\circ\beta_{w_n})(z)|\geq\gamma_{N+1} \ \ 
\ \text{ for every }\ \  z\in \mathbb  D, \ \  |z|\geq a. \end{equation} 

By Lemma \ref{lem36}, there exists $w_{N+1}\in[r_{N+1},1)$ such that 
 \begin{equation} \label{310} |(B_{N+1}\circ\beta_{w_{N+1}})(z)|\geq \gamma_{N+1}
 \ \  \text{ for every } z\in \mathbb  D, \ \
|z|\leq a.  \end{equation} 
We show that 
 \begin{equation} \label{311} |(\theta_N\cdot B_{N+1}\circ\beta_{w_{N+1}})(z)|\geq
\gamma_{N+1}\delta_N\inf_{1\leq n\leq N+1}|(B_n\circ\beta_{w_n})(z)| \ \ 
\text{ for every } z\in \mathbb  D.  \end{equation} 
We consider four cases.

\noindent {\it First case:} $a\leq |z|<1$, 
$\inf_{1\leq n\leq N+1}|(B_n\circ\beta_{w_n})(z)|=
\inf_{1\leq n\leq N}|(B_n\circ\beta_{w_n})(z)|$.
By  \eqref{37} and  \eqref{39}, 
$$|(\theta_N\cdot B_{N+1}\circ\beta_{w_{N+1}})(z)|\geq 
\gamma_{N+1}\delta_N\inf_{1\leq n\leq N}|(B_n\circ\beta_{w_n})(z)|$$
$$ = 
\gamma_{N+1}\delta_N\inf_{1\leq n\leq N+1}|(B_n\circ\beta_{w_n})(z)|.$$

\noindent {\it Second case:} $a\leq |z|<1$, 
$\inf_{1\leq n\leq N+1}|(B_n\circ\beta_{w_n})(z)|=
|(B_{N+1}\circ\beta_{w_{N+1}})(z)|$. By  \eqref{38},
$$|(\theta_N\cdot B_{N+1}\circ\beta_{w_{N+1}})(z)|
\geq\gamma_{N+1}|(B_{N+1}\circ\beta_{w_{N+1}})(z)|
= \gamma_{N+1}\!\inf_{1\leq n\leq N+1}\!|(B_n\circ\beta_{w_n})(z)|.$$

\noindent {\it Third case:} $|z|\leq a$, 
$\inf_{1\leq n\leq N+1}|(B_n\circ\beta_{w_n})(z)|=
\inf_{1\leq n\leq N}|(B_n\circ\beta_{w_n})(z)|$. By  \eqref{37} and  \eqref{310}, 
$$|(\theta_N \cdot B_{N+1}\circ\beta_{w_{N+1}})(z)|\geq
|(B_{N+1}\circ\beta_{w_{N+1}})(z)|\delta_N
\inf_{1\leq n\leq N}|(B_n\circ\beta_{w_n})(z)|$$
$$\geq
\gamma_{N+1}\delta_N
\inf_{1\leq n\leq N}|(B_n\circ\beta_{w_n})(z)|=
\gamma_{N+1}\delta_N
\inf_{1\leq n\leq N+1}|(B_n\circ\beta_{w_n})(z)|.$$

\noindent {\it Fourth case:} $|z|\leq a$, 
$\inf_{1\leq n\leq N+1}|(B_n\circ\beta_{w_n})(z)|= 
|(B_{N+1}\circ\beta_{w_{N+1}})(z)|$. 
By  \eqref{37} and  \eqref{310},
$$|(\theta_N\cdot B_{N+1}\circ\beta_{w_{N+1}})(z)|
\geq|(B_{N+1}\circ\beta_{w_{N+1}})(z)|\delta_N
\inf_{1\leq n\leq N}|(B_n\circ\beta_{w_n})(z)| $$
$$ \geq
|(B_{N+1}\circ\beta_{w_{N+1}})(z)|\delta_N\gamma_{N+1}=
\gamma_{N+1}\delta_N\inf_{1\leq n\leq N+1}|(B_n\circ\beta_{w_n})(z)|.$$

Since $\delta_{N+1}=\gamma_{N+1}\delta_N$ and $0<\delta_N<1$,
the relation \eqref{311} is proved. Thus, \eqref{37} is proved by induction
for all $N$. 
Let $N\to\infty$ in both parts of  \eqref{37}, then \eqref{36} follows.
By Lemma  \ref{lem31}, there exists a sequence 
$\zeta_N\in\mathbb  T$ such that 
the product $\prod_N\zeta_NB_N\circ\beta_{w_N}$ converges.

Let $f_N$ be  functions from Proposition  \ref{prop37}  
applied to $C$, $B_N$, $\varphi_N$, and $g$, with $w=w_N$.
Put $\psi_N=f_N\circ\beta_{w_N}$. Then $\psi_N \in H^\infty$, 
$g -\psi_N\cdot \varphi_N\circ\beta_{w_N}\in (B_N\circ\beta_{w_N})H^\infty$
and $\operatorname{dist}(\psi_N,(B_N\circ\beta_{w_N})H^\infty)
=\operatorname{dist}(f_N,B_N H^\infty)\leq C$. \end{proof}

 \section{Preliminaries: Jordan operators }

For $N\geq 1$ and $\lambda\in\mathbb  D$ define the operator
$T_{N,\lambda}\colon\mathbb  C^{N+1}\to\mathbb  C^{N+1}$ acting 
by the formula $T_{N,\lambda}e_0 = \lambda e_0$, 
$T_{N,\lambda}e_n = \lambda e_n + e_{n-1}$, $1\leq n\leq N$, 
where $\{e_n\}_{n=0}^N$ is an orthonormal basis of $\mathbb  C^{N+1}$. 
It is well known that $T_{N,\lambda}\approx  T_{b_\lambda^{N+1}}$, 
but $\|T_{N,\lambda}\|\leq 1$ if and only if $\lambda = 0$. 

\begin{lemma}\label{lem41} Let $N\geq 1$, and let $\varepsilon >0$. 
Then there exists $\delta > 0$ such that for every 
$\lambda\in\mathbb  D$, $|\lambda|<\delta$, there exists an operator 
$X\colon\mathbb  C^{N+1}\to\mathcal K_{b_\lambda^{N+1}}$ 
such that $X T_{N,\lambda} = T_{b_\lambda^{N+1}} X$,
$\|X\|\leq 1 + \varepsilon$, and $\|X^{-1}\|\leq 1 + \varepsilon$, 
where the space $\mathcal K_{b_\lambda^{N+1}}$ and the operator 
$T_{b_\lambda^{N+1}}$ are defined in \eqref{17}.\end{lemma}

\begin{proof}  Put 
$$\textstyle h_n(z) = b_\lambda^{N+1}(z)
\frac{(1-|\lambda|^2)^{1/2}(1-\overline\lambda z)^n}{(z-\lambda)^{n+1}}, 
 \ z\in\mathbb  D, 
\ \text{ and } \  c_{nk}=\frac{n!}{k!(n-k)!}
\frac{\overline\lambda^{n-k}}{(1-|\lambda|^2)^n},$$
$0\leq k\leq n$, $0\leq n\leq N$. 
Then $\{h_n\}_{n=0}^N$ is an orthonormal basis of $\mathcal K_{b_\lambda^{N+1}}$. 
Put 
$$ X\colon\mathbb  C^{N+1}\to \mathcal K_{b_\lambda^{N+1}}, \ \ 
Xe_n  = \sum_{k=0}^nc_{nk}h_k.$$
Then $X$ has the  upper triangular form 
in the orthonormal bases  $\{e_n\}_{n=0}^N$ and $\{h_n\}_{n=0}^N$. 
Also, $X T_{N,\lambda} = T_{b_\lambda^{N+1}} X$.
To see that, put $u_n=Xe_n$, 
then $u_n(z) = b_\lambda^{N+1}(z)
\frac{(1-|\lambda|^2)^{1/2}}{(z-\lambda)^{n+1}}$, $z\in\mathbb  D$, 
therefore, $T_{b_\lambda^{N+1}}u_0 = \lambda u_0$, 
$T_{b_\lambda^{N+1}}u_n = \lambda u_n + u_{n-1}$, $1\leq n\leq N$.

Let $D$ be a diagonal matrix with the elements $c_{nn}$, 
$0\leq n\leq N$, put $A=X-D$, then $(D^{-1}A)^{N+1} = \mathbb  O$, 
therefore, $X^{-1} = \sum_{n=0}^N (-1)^n(D^{-1}A)^nD^{-1}$. 
Also, $\|D\| = \frac{1}{(1-|\lambda|^2)^N}$, and 
$\|D^{-1}\| = 1$.
If $0\leq k\leq n-1$, then $c_{nk}\to 0$ when $|\lambda|\to 0$, 
and, since $A$ is a finite matrix with elements $c_{nk}$ and $0$, 
then $\|A\|\to 0$ when $|\lambda|\to 0$. 
Since $\|X\|\leq \|D\| + \|A\|$ 
and $\|X^{-1}\| \leq \sum_{n=0}^N \|A\|^n$,
$\|X\|\to 1$ and $\|X^{-1}\|\to 1$ when $|\lambda|\to 0$.
Therefore, there exists $\delta > 0$ such that if 
$|\lambda|<\delta$ then $\|X\|\leq 1 + \varepsilon$, 
and $\|X^{-1}\|\leq 1 + \varepsilon$. \end{proof} 

Suppose $\mathcal H$ is a Hilbert space, and $N\geq 1$. 
Define the shift operator $\text{\bf S}_{\mathcal H N}$ 
on $\oplus_{n=0}^N\mathcal H$ as following:
 \begin{equation} \label{41} \text{\bf S}_{\mathcal H N} = \begin{pmatrix}
\mathbb  O & I_{\mathcal H} &  \ldots & \mathbb  O & \mathbb  O\\ 
\ldots & \ldots & \ldots & \ldots & \ldots \\ 
\mathbb  O & \mathbb  O  & \ldots & \mathbb  O & I_{\mathcal H} \\ 
\mathbb  O & \mathbb  O  & \ldots & \mathbb  O   & \mathbb  O 
\end{pmatrix}.  \end{equation} 
Note that $\text{\bf S}_{\mathcal H N}^{N+1} = \mathbb  O$, and  
$\text{\bf S}_{\mathcal H N}^N \neq \mathbb  O$. Also, 
$\mu_{\text{\bf S}_{\mathcal H N}}=\dim \mathcal H$.

\begin{corollary}\label{cor42} Suppose $\mathcal H$ is a Hilbert space, 
$\text{\bf d} = \dim \mathcal H <\infty$, $\{e_j\}_{j=1}^{\text{\bf d}}$ 
is an orthonormal basis of $\mathcal H$,  
$\{\lambda_j\}_{j=1}^{\text{\bf d}}\subset \mathbb  D$, 
the operator $D\colon\mathcal H\to\mathcal H$ acts by the formula 
$De_j=\lambda_je_j$, $1\leq j\leq \text{\bf d}$.
Furthermore, suppose 
$N\geq 1$, and put $T=\oplus_{n=0}^N D + \text{\bf S}_{\mathcal H N}$.
Then $T\approx \oplus_{j=1}^{\text{\bf d}}T_{b_{\lambda_j}^{N+1}}$.
Moreover, let $\varepsilon >0$, and let $\delta$ be from Lemma \ref{lem41} 
applied to $N$ and $\varepsilon$. If $|\lambda_j|<\delta$ 
for every $1\leq j\leq \text{\bf d}$, then 
 there exists an operator 
$X\colon\oplus_{n=0}^N\mathcal H\to\oplus_{j=1}^{\text{\bf d}}
\mathcal K_{b_{\lambda_j}^{N+1}}$ 
such that $X T = (\oplus_{j=1}^{\text{\bf d}}T_{b_{\lambda_j}^{N+1}}) X$,
$\|X\|\leq 1 + \varepsilon$, and $\|X^{-1}\|\leq 1 + \varepsilon$.\end{corollary}

\begin{proof} We have $\mathcal H =\oplus_{j=1}^{\text{\bf d}}\mathbb  C e_j$, 
therefore, 
$\oplus_{n=0}^N\mathcal H = \oplus _{j=1}^{\text{\bf d}}\oplus_{n=0}^N\mathbb  C e_j$.
The spaces $\oplus_{n=0}^N\mathbb  C e_j$ are invariant for $T$, and 
$T|_{\oplus_{n=0}^N\mathbb  C e_j} = T_{N,\lambda_j}$. Thus, 
$T=\oplus_{j=1}^{\text{\bf d}}T_{N,\lambda_j}$, and it remains to apply Lemma \ref{lem41}.\end{proof}

\begin{lemma}\label{lem43} Suppose $\mathcal H$ is a Hilbert space, 
$\text{\bf d} = \dim \mathcal H <\infty$, $\{e_j\}_{j=1}^{\text{\bf d}}$ 
is an orthonormal basis of $\mathcal H$,  
$\{\lambda_j\}_{j=1}^{\text{\bf d}}$, $\{\nu_j\}_{j=1}^{\text{\bf d}}\subset \mathbb  D$, 
the operators $D\colon\mathcal H\to\mathcal H$, 
$D_\star\colon\mathcal H\to\mathcal H$ act by the formulas 
$De_j=\lambda_je_j$, $1\leq j\leq \text{\bf d}$, 
$D_\star e_j=\nu_je_j$, $1\leq j\leq \text{\bf d}$.
Let $N\geq 1$, and let 
$A\colon\oplus_{n=0}^N\mathcal H\to\oplus_{n=0}^N\mathcal H$ 
be an arbitrary operator. Put 
$B=\prod_{j=1}^{\text{\bf d}}b_{\lambda_j}^{N+1} b_{\nu_j}^{N+1}$.
Put $T_0 = \oplus_{n=0}^N D + \text{\bf S}_{\mathcal H N}$ and 
$T_1  = \oplus_{n=0}^N D_\star + \text{\bf S}_{\mathcal H N}^\ast $.
Define the operator $T$ 
on the space $\oplus_{n=0}^N\mathcal H\oplus\oplus_{n=0}^N\mathcal H$
as follows:
$$T = \begin{pmatrix} T_1 & A \\  \mathbb  O & T_0
 \end{pmatrix}.$$
If $\lambda_j\neq\lambda_k$, $\nu_j\neq\nu_k$ for $j\neq k$, and 
$\lambda_j\neq \nu_k$ for every $1\leq j, k\leq \text{\bf d}$,
then $T\approx T_B$, where $T_B$ is defined in \eqref{17}.\end{lemma}

\begin{proof}
Put $p(z)=\prod_{j=1}^{\text{\bf d}}(z-\lambda_j)^{N+1}$.
Then $p(T_0) = \mathbb  O$, $p(T_1)$ has bounded inverse,
because $\lambda_j\neq \nu_k$ for every $1\leq j, k\leq \text{\bf d}$
and $T_1$ acts on a finite dimensional space, and
$$p(T) = \begin{pmatrix} p(T_1) & A_0\\  \mathbb  O & \mathbb  O
 \end{pmatrix},$$
where $A_0$ is an appropriate operator. Put $Y_0=p(T_1)^{-1}A_0$
and $Y = \begin{pmatrix} I & Y_0\\  \mathbb  O & I \end{pmatrix}$.
The equality $YT = (T_1\oplus T_0)Y$ follows from the definition of $Y$ 
and the equality $p(T)T = Tp(T)$ writing in matrix form.
Thus, $T\approx T_1\oplus T_0$. By Corollary \ref{cor42},
$T_0\approx \oplus_{j=1}^{\text{\bf d}}T_{b_{\lambda_j}^{N+1}}$ 
and $T_1\approx \oplus_{j=1}^{\text{\bf d}}T_{b_{\nu_j}^{N+1}}$. 
Since $B$ is a finite Blaschke product and all $\lambda_j$, $\nu_k$ are
pairwise distinct,  we conclude that 
$T\approx T_B$. \end{proof}

\section{Preliminaries: Foguel--Hankel operators with truncated shifts}

We introduce the following notation (see [DP], Sec. 1.3]):
$$\text{\bf V} = \begin{pmatrix} 1 & 0 \\  0 & -1 \end{pmatrix}, \ \ \ 
\text{\bf C} = \begin{pmatrix} 0 & 0 \\  1 & 0 \end{pmatrix}, \ \ \ 
\text{\bf I}_2 = \begin{pmatrix} 1 & 0 \\  0 & 1 \end{pmatrix}, \ \ \ 
\text{\bf D}(a, c) = \begin{pmatrix} a & 0 \\  0 & c \end{pmatrix}, \ \ 
a,c \in \mathbb  C.$$
For $N\geq 1$ put 
 \begin{equation} \label{51}\text{\bf C}_{Nj} = \text{\bf V}^{\otimes j}\otimes \text{\bf C}
\otimes\text{\bf I}_2^{\otimes N-j}, \ \ 0\leq j\leq N,  \end{equation} 
and for the families $\{a_l\}_{l=0}^N$, $\{c_l\}_{l=0}^N\subset\mathbb  C$
put 
 \begin{equation} \label{52} \text{\bf D}(\{a_l\}_{l=0}^N,\{c_l\}_{l=0}^N) = 
\text{\bf D}(a_0, c_0)\otimes\ldots\otimes \text{\bf D}(a_N, c_N)  \end{equation}
 and 
 \begin{equation} \label{53}\begin{aligned}& \text{\bf D}_j(\{a_l\}_{l=0}^N,\{c_l\}_{l=0}^N) \\ &
  = 
\text{\bf D}(a_0, c_0)\otimes\ldots\otimes 
\text{\bf D}(a_{j-1}, c_{j-1})\otimes\text{\bf D}(c_j,a_j)
\otimes \text{\bf D}(a_{j+1}, c_{j+1})\otimes\ldots\otimes
\text{\bf D}(a_N, c_N), \end{aligned} \end{equation} 
$$  0\leq j\leq N. $$
Here $A^{\otimes j}$ denotes the tensor product of $j$ copies of $A$. 
$\text{\bf C}_{Nj}$, $\text{\bf D}(\{a_l\}_{l=0}^N,\{c_l\}_{l=0}^N)$ and
$\text{\bf D}_j(\{a_l\}_{l=0}^N,\{c_l\}_{l=0}^N)$ are 
$2^{N+1}\times 2^{N+1}$ matrices, that is, operators on $\mathbb  C^{2^{N+1}}$, 
and $\text{\bf D}(\{a_l\}_{l=0}^N,\{c_l\}_{l=0}^N)$ and
$\text{\bf D}_j(\{a_l\}_{l=0}^N,\{c_l\}_{l=0}^N)$, $0\leq j\leq N$,
 are diagonal
with respect to the standard basis in $\mathbb  C^{2^{N+1}}$. 
From the equalities 
$$\text{\bf V}\text{\bf D}(a, c) = \text{\bf D}(a, c)\text{\bf V}, \ \ 
\text{\bf I}_2\text{\bf D}(a, c) = \text{\bf D}(a, c)\text{\bf I}_2  
\ \text{ and } \ 
\text{\bf C}\text{\bf D}(a, c) = \text{\bf D}(c, a) \text{\bf C}$$
we conclude that 
 \begin{equation} \label{54}  \text{\bf C}_{Nj}\text{\bf D}(\{a_l\}_{l=0}^N,\{c_l\}_{l=0}^N) = 
\text{\bf D}_j(\{a_l\}_{l=0}^N,\{c_l\}_{l=0}^N)\text{\bf C}_{Nj}, \ \ 
0\leq j\leq N.  \end{equation} 

\begin{lemma}\label{lem51} For every $0<\delta<1$ there exist 
families $\{a_l\}_{l=0}^N$, $\{c_l\}_{l=0}^N\subset(0,\delta^{1/(N+1)})$
such that the elements of a diagonal matrix
$\text{\bf D}(\{a_l\}_{l=0}^N,\{c_l\}_{l=0}^N)$ are from $(0,\delta)$ and
are pairwise distinct. \end{lemma}

\begin{proof} The elements of a diagonal matrix 
$\text{\bf D}(\{a_l\}_{l=0}^N,\{c_l\}_{l=0}^N)$ are the products 
of $N+1$ factors, each of which is equal to $a_l$ or $c_l$ for some
$l$, $0\leq l \leq N$. 
Therefore, if $0<a_l<\delta^{1/(N+1)}$ and 
$0<c_l<\delta^{1/(N+1)}$ for $0\leq l\leq N$, then 
the elements of a diagonal matrix
$\text{\bf D}(\{a_l\}_{l=0}^N,\{c_l\}_{l=0}^N)$ are from $(0,\delta)$.

The choice of $\{a_l\}_{l=0}^N$, $\{c_l\}_{l=0}^N\subset(0,\delta^{1/(N+1)})$ is 
such that the elements of a diagonal matrix $\text{\bf D}(\{a_l\}_{l=0}^N,\{c_l\}_{l=0}^N)$ 
are pairwise distinct by induction. Base of induction: $a_0\neq c_0$. 
Suppose $1\leq n\leq N$, and suppose that $\{a_l\}_{l=0}^n$, 
$\{c_l\}_{=0}^n\subset(0,\delta^{1/(N+1)})$ 
are such that the elements $\{d_k\}_{k=1}^{2^{n+1}}$
of the matrix $\text{\bf D}(a_0, c_0)\otimes\ldots\otimes
 \text{\bf D}(a_n, c_n)$ are pairwise distinct.
Let $a_{n+1}$, $c_{n+1}\subset(0,\delta^{1/(N+1)})$ 
be such that  $c_{n+1}<a_{n+1} d_k$ for all $k$, $1\leq k\leq 2^{n+1}$.
The elements of $\text{\bf D}(a_0, c_0)\otimes\ldots\otimes
 \text{\bf D}(a_{n+1}, c_{n+1})$ are 
$a_{n+1} d_k$ and $c_{n+1} d_k$, $1\leq k\leq 2^{n+1}$.
It is easy to see that  $\{a_l\}_{l=0}^{n+1}$ and 
$\{c_l\}_{l=0}^{n+1}$ satisfy the inductional assumption again. \end{proof}

Suppose $\mathcal H$ is a Hilbert space,  $N\geq 1$, and 
$A_j\colon\mathcal H\to\mathcal H$, $0\leq j\leq N$, are operators. 
Define  
a Hankel operator $\Gamma(\{A_j\}_{j=0}^N)$ 
on $\oplus_{j=0}^N\mathcal H$ as follows: 
 \begin{equation} \label{55} 
\Gamma(\{A_j\}_{j=0}^N) = \begin{pmatrix} 
\mathbb  O & \mathbb  O &  \ldots & \mathbb  O & A_N \\  
\mathbb  O & \mathbb  O &  \ldots & A_N &  A_{N-1} \\  
\ldots & \ldots & \ldots & \ldots & \ldots \\  
\mathbb  O & A_N & \ldots & A_{N-1} & A_1 \\  
 A_N & A_{N-1}& \ldots & A_1 & A_0 \end{pmatrix}.  \end{equation} 
Put $\Gamma(k,\{A_j\}_{j=0}^N)= 
\Gamma(A_k, \ldots, A_N, \mathbb  O, \ldots, \mathbb  O)$, $0\leq k\leq N$. 
Then  
 \begin{equation} \label{56} \text{\bf S}_{\mathcal H N}^\ast \Gamma(\{A_j\}_{j=0}^N)= 
\Gamma(\{A_j\}_{j=0}^N)\text{\bf S}_{\mathcal H N} \ \  \text{  and } \ \  
\Gamma(\{A_j\}_{j=0}^N)\text{\bf S}_{\mathcal H N}^k=\Gamma(k,\{A_j\}_{j=0}^N),  \end{equation} 
where  $\text{\bf S}_{\mathcal H N}$ is defined in  \eqref{41}.

\smallskip

Define the operator $Q_N(\{A_j\}_{j=0}^N)$ on 
$\bigoplus_{j=0}^N\mathcal H\oplus\bigoplus_{j=0}^N\mathcal H$
as follows:
 \begin{equation} \label{57}  Q_N(\{A_j\}_{j=0}^N) = \begin{pmatrix} \text{\bf S}_{\mathcal H N}^\ast & 
\Gamma(\{A_j\}_{j=0}^N) \\  
\mathbb  O & \text{\bf S}_{\mathcal H N}\end{pmatrix}.  \end{equation} 
Operators from \eqref{57} are analogs of 
 {\it Foguel--Hankel} operators 
 (truncated shifts are used in the construction
 instead of the forward and backward shifts).
 We will call such operators ``truncated" Foguel--Hankel operators
 in this paper.
It is easy to see that   $$(Q_N(\{A_j\}_{j=0}^N))^{N+2} = \mathbb  O \ \ 
\text{ and } \ \ (Q_N(\{A_j\}_{j=0}^N))^N \neq \mathbb  O.$$
  (Note that
$(Q_N(\{A_j\}_{j=0}^N))^{N+1} = \mathbb  O$
if and only if $A_N=\mathbb  O$). Also, 
$$\mu_{Q_N(\{A_j\}_{j=0}^N)}\geq \mu_{\text{\bf S}_{\mathcal H N}}=\dim \mathcal H. $$ 

The following is proved in [DP, Theorem 3.1 and Remark 3.6], 
see also [Pi], [Pe, Ch.15.3]. 
Let $\{\alpha_{Nj}\}_{j=0}^N\subset \mathbb  C$.
Put  \begin{equation} \label{58} R_N(\{\alpha_{Nj}\}_{j=0}^N) = 
Q_N(\{\alpha_{Nj}\text{\bf C}_{N j}\}_{j=0}^N),  \end{equation} 
where $Q_N$ is defined in  \eqref{57}, and $\text{\bf C}_{N j}$ are defined in  \eqref{51}. 
The operator $R_N$ acts on $\bigoplus_{j=0}^N\mathbb  C^{2^{N+1}}\oplus\bigoplus
_{j=0}^N\mathbb  C^{2^{N+1}}$.
Then 
$$M_{pb}(R_N)^2\leq
3(4\sup_{1\leq n\leq N+1} n^2\sum_{j=n-1}^N|\alpha_{Nj}|^2 + 1)$$
and $$M_{cpb}(R_N)^2\geq \frac{1}{4}\sum_{j=1}^{N+1} j^2|\alpha_{N,j-1}|^2.$$
Thus, if $\{\alpha_{Nj}\}_{j=0}^N$, $N\geq 1$, are such that 
 \begin{equation} \label{59} \sup_N \sup_{1\leq n\leq N+1} n^2\sum_{j=n-1}^N|\alpha_{Nj}|^2 < \infty
\ \ \text{ and } \ \ \sup_N\sum_{j=1}^{N+1}j^2|\alpha_{N,j-1}|^2=\infty, \end{equation} 
 then 
$\sup_N M_{pb}(R_N)< \infty$ and $\sup_N M_{cpb}(R_N)= \infty$.
For example, one can take $\alpha_{Nj}=\alpha_j=(j+1)^{-3/2}$, $0\leq j\leq N$, 
for every $N$.

Let $\{\alpha_{Nj}\}_{j=0}^N\subset \mathbb  C$, $N\geq 1$, satisfy \eqref{59}, 
and 
let $R_N = R_N(\{\alpha_{Nj}\}_{j=0}^N)$ be defined in \eqref{58}.
By Proposition \ref{prop11}, $R = \oplus_N R_N$ 
is polynomially bounded 
and is not similar to a contraction. The minimal function of $R$ is
the least common multiple of the minimal functions of $R_N$, 
that is,  of $\chi^{N+2}$ or $\chi^{N+1}$, where $\chi(z)=z$, $z\in\mathbb  D$, 
therefore,  the zero function.
Thus, $R$ is not a $C_0$-operator.
Let $\{w_N\}_N\subset\mathbb  D$ be such that $\sum_N (N+2)(1-|w_N|)<\infty$.
Set $B = \prod_Nb_{w_N}^{N+2}$ and $R = \oplus_N\beta_{w_N}(R_N)$, 
where $\beta_w$ is defined in \eqref{13} and $b_w$ is a Blaschke factor.
By Corollary \ref{cor12}, $R$ 
is polynomially bounded 
and is not similar to a contraction. 
Also, $B(R)=\mathbb  O$. 
But $\mu_R\geq\mu_{\beta_{w_N}(R_N)}=\mu_{R_N}\geq 2^{N+1}$ for every $N$
(see \eqref{15} and \eqref{16}),
therefore, $\mu_R=\infty$.

\section{ Perturbation of ``truncated" Foguel--Hankel operators}

Before  formulating the main result of this section, 
we introduce the following notation.
Set $q_n(\lambda,\nu) =
\sum_{k=0}^{n-1}\lambda^{n-1-k}\nu^k$, $n\geq 1$, $\lambda$, 
$\nu\in\mathbb  C$, $q_0=0$.
If $\lambda\neq\nu$, then $q_n(\lambda,\nu) = \frac{\lambda^n - \nu^n}{\lambda - \nu}$.
Furthermore, 
$$ \lambda^n + q_n(\lambda,\nu)\nu = q_{n+1}(\lambda,\nu),$$ 
 \begin{equation} \label{61} \begin{aligned}
&\textstyle{\frac{n!}{k!(n-k)!}(\lambda^{n-k} + q_{n-k}(\lambda,\nu)\nu)=
 \frac{n!}{k!(n-k)!}q_{n+1-k}(\lambda,\nu)} \\ 
&=\textstyle{\frac{(n+1)!}{k!(n+1-k)!}q_{n+1-k}(\lambda,\nu)-\textstyle{\frac{n!}{(k-1)!(n+1-k)!}}q_{n+1-k}(\lambda,\nu), \ \ \ 1\leq k\leq n.}  \end{aligned}\end{equation} 

\begin{lemma}\label{lem61} Suppose $N\geq 1$, $0<\delta<1$, 
$p$ is (an analytic) polynomial,
and $\widehat p(n)=0$ for $0\leq n\leq N+1$.
Then
$$\biggl|\sum_{n\geq N+2} \textstyle{
\widehat p(n)\frac{n!}{k!(n-k)!}q_{n-k}(\lambda, \nu)\biggr|
\leq \frac{(k+1)(k+2)}{(1-\delta)^{(k+3)}}\delta\|p\|_\infty}$$
for $\lambda$, $\nu\in\mathbb  C$, $|\lambda|\leq\delta$, 
$|\nu|\leq\delta$, $0\leq k\leq N$.\end{lemma}

\begin{proof} We have $$\sum_{n\geq N+2} 
\widehat p(n)\textstyle{\frac{n!}{k!(n-k)!}q_{n-k}}(\lambda, \nu)
=
\frac{1}{k!}\frac{p^{(k)}(\lambda)-p^{(k)}(\nu)}
{\lambda - \nu}, \ \ \text{ if } \lambda\neq\nu, $$
 and  $$\sum_{n\geq N+2} 
\widehat p(n)\textstyle\frac{n!}{k!(n-k)!}q_{n-k}(\lambda, \lambda) = 
\frac{1}{k!}p^{(k+1)}(\lambda). $$
If $0<\delta<1$ and $|\lambda|<\delta$, $|\nu|<\delta$, $\lambda\neq\nu$,
then $$\biggl|\frac{p^{(k)}(\lambda)-p^{(k)}(\nu)}
{\lambda - \nu}\biggr|\leq\sup_{|z|\leq\delta}|p^{(k+1)}(z)|
 \leq\sup_{|z|\leq\delta}|p^{(k+1)}(z)-p^{(k+1)}(0)|$$
$$ 
\leq\sup_{|z|\leq\delta}|p^{(k+2)}(z)|\delta
\leq \frac{(k+2)!\delta}{(1-\delta)^{(k+3)}}\|p\|_\infty
\ \  \text{ for } 0\leq k\leq N, $$
because $p^{(k+1)}(0) = 0$ for $0\leq k\leq N$.

If $\lambda=\nu$, the conclusion of the lemma follows from the estimate 
of $p^{(k+1)}$ above. \end{proof}

The following theorem is the main result of this section.

\begin{theorem}\label{thm62} Suppose $N\geq 1$, $\mathcal H$ is a Hilbert space, 
$D$, $D_\star$, $D_j$, $A_j$ are operators on $\mathcal H$
such that $D_jA_j = A_jD$, and $D_\star$, $D_j$  for  $0\leq j\leq N$
are simultaneously diagonalizable with respect to some orthonormal basis.
Define the operators $R$ and $T$ on the space 
$\bigoplus_{n=0}^N\mathcal H\oplus\bigoplus_{n=0}^N\mathcal H$ as follows:
$$ R = Q_N (\{A_j\}_{j=0}^N) = 
\begin{pmatrix} \text{\bf S}_{\mathcal H N}^\ast & \Gamma(\{A_j\}_{j=0}^N) \\ 
\mathbb  O & \text{\bf S}_{\mathcal H N} \end{pmatrix} $$
and
$$ T = R + \textstyle\bigoplus_{n=0}^N D_\star \oplus\bigoplus _{n=0}^N D =
\begin{pmatrix} \bigoplus_{n=0}^N D_\star + \text{\bf S}_{\mathcal H N}^\ast & 
\Gamma(\{A_j\}_{j=0}^N) \\ 
\mathbb  O & \bigoplus_{n=0}^N D + 
 \text{\bf S}_{\mathcal H N} \end{pmatrix}, $$
where 
$\Gamma(\{A_j\}_{j=0}^N)$ is defined in  \eqref{55} and 
$\text{\bf S}_{\mathcal H N}$ is defined in \eqref{41}.
For a polynomial $p$ put 
$$\textstyle R_{(p)}\! =\! P_{\bigoplus_{n=0}^N\mathcal H\oplus\{0\}} 
p(R)|_{\{0\}\oplus\bigoplus_{n=0}^N\mathcal H}
\  \text{ and }  \ 
T_{(p)}\! =\! P_{\bigoplus_{n=0}^N\mathcal H\oplus\{0\}}
p(T)|_{\{0\}\oplus\bigoplus_{n=0}^N\mathcal H},$$
that is, $R_{(p)}$ and $T_{(p)}$ are the operators 
from the right upper corner in the matrix forms of $p(R)$ and $p(T)$,
respectively.

Let $\varepsilon > 0$. Then there exists $\delta>0$ which depends on $N$,
$\{A_j\}_{j=0}^N$, and $\varepsilon$, such that if $\|D_\star\|<\delta$ 
and $\|D_j\|<\delta$, $0\leq j\leq N$, then 
$$\|T_{(p)} - R_{(p)}\|\leq\varepsilon\|p\|_\infty \  
\text{ for every polynomial }\ p.$$\end{theorem}

\begin{proof} Recall that $\chi(z)=z$, $z\in\mathbb  D$. To simplify notation set 
$\text{\bf S} = \text{\bf S}_{\mathcal H N}$,
$K_\star = \oplus_{n=0}^N D_\star$, $K = \oplus_{n=0}^N D$, 
$T_{(n)} = T_{(\chi^n)}$, $R_{(n)} = R_{(\chi^n)}$. 
We have $T_{(0)} =R_{(0)} =\mathbb  O$ and  $T_{(1)} =R_{(1)} = \Gamma(\{A_j\}_{j=0}^N)$.

Since  $\text{\bf S}^\ast K_\star =K_\star\text{\bf S}^\ast$ and
 $\text{\bf S}^{N+1}=\mathbb  O$,
 \begin{equation} \label{62} (\text{\bf S}^\ast + K_\star)^n = 
\sum_{k=0}^{\min(n, N)}\textstyle\frac{n!}{k!(n-k)!}K_\star^{n-k}(\text{\bf S}^\ast)^k.  \end{equation}
If the operators $L_j$ on $\mathcal H$ are such that 
 $D_jL_j = L_jD$, $0\leq j\leq N$, then
 \begin{equation} \label{63} \Gamma(n,\{L_j\}_{j=0}^N)K = \Gamma(n,\{D_jL_j\}_{j=0}^N) 
\ \ \text{ for } n\geq 0. \end{equation}
Also, for arbitrary operators $L_j$ on $\mathcal H$, $0\leq j\leq N$, 
 \begin{equation} \label{64} K_\star^k\Gamma(n,\{L_j\}_{j=0}^N)= \Gamma(n,\{D_\star^k L_j\}_{j=0}^N) 
\ \ \text{ for } n\geq 0, \ k\geq 0.  \end{equation}
From the equality 
$R_{(n)}=\sum_{k=0}^{n-1}(\text{\bf S}^\ast)^{n-1-k}\Gamma(\{A_j\}_{j=0}^N)\text{\bf S}^k$ 
and  \eqref{56}, taking into account that $\text{\bf S}^{N+1}=\mathbb  O$,
we obtain that 
 \begin{equation} \label{65} R_{(n)} = \Gamma(n-1,\{nA_j\}_{j=0}^N), \ \ 1\leq n\leq N+1, 
\ \ \text{ and } R_{(n)} = \mathbb  O, \ \ n\geq N+2.   \end{equation}

Since $D_\star$ and 
 $D_j$  for  $0\leq j\leq N$
are simultaneously diagonalizable with respect to some orthonormal basis,
 we have that
$D_\star D_j = D_j D_\star$, $0\leq j\leq N$, therefore, 
$q_n(D_\star, D_j)$ is well defined and \eqref{61} is fulfilled for $D_\star$ and $D_j$ 
instead of $\lambda$ and $\nu$.
We prove by induction that 
 \begin{equation} \label{66} T_{(n)} = \sum_{k=0}^{\min(n-1,N)}
\Gamma(k,\{\textstyle\frac{n!}{k!(n-k)!}q_{n-k}(D_\star, D_j)A_j\}_{j=0}^N), 
\ \  n\geq 1. \end{equation}

Base of induction: if $n=1$, then  \eqref{66} is clearly fulfilled.
Assume that   \eqref{66}  is fulfilled for some $n$, $1\leq n\leq N$. 
We have from  \eqref{56}, \eqref{61}, \eqref{62},  \eqref{63}, and \eqref{64} that
$$T_{(n+1)} = (\text{\bf S}^\ast + K_\star)^n\Gamma(\{A_j\}_{j=0}^N) + 
T_{(n)}(\text{\bf S} + K) $$
$$=
\biggl(\sum_{k=0}^n\textstyle\frac{n!}{k!(n-k)!}
K_\star^{n-k}(\text{\bf S}^\ast)^k\biggr)
\Gamma(\{A_j\}_{j=0}^N)$$
$$ + 
\sum_{k=0}^{n-1}
\Gamma(k,\{\textstyle\frac{n!}{k!(n-k)!}q_{n-k}(D_\star, D_j)A_j\}_{j=0}^N)
(\text{\bf S}+K)$$
$$ = \sum_{k=0}^n
\Gamma(k,\{\textstyle\frac{n!}{k!(n-k)!}D_\star^{n-k}A_j\}_{j=0}^N) +
\displaystyle{\sum_{k=0}^{n-1}}
\Gamma(k+1,\{\textstyle\frac{n!}{k!(n-k)!}q_{n-k}(D_\star, D_j)A_j\}_{j=0}^N)$$
$$ +
\sum_{k=0}^{n-1}\Gamma(k,\{\textstyle\frac{n!}{k!(n-k)!}q_{n-k}(D_\star, D_j)D_jA_j\}_{j=0}^N)$$
$$ = 
\Gamma(\{D_\star^nA_j\}_{j=0}^N) + \Gamma(\{q_n(D_\star,D_j)D_jA_j\}_{j=0}^N)
 +\Gamma(n,\{A_j\}_{j=0}^N) +
\Gamma(n, n\{A_j\}_{j=0}^N)$$
$$ +\sum_{k=1}^{n-1}
\Gamma\biggl(k,\Bigl\{\Bigl(\textstyle\frac{n!}{k!(n-k)!}D_\star^{n-k} $$
$$+ 
\textstyle\frac{n!}{(k-1)!(n-k+1)!}q_{n-k+1}(D_\star, D_j)+
\frac{n!}{k!(n-k)!}q_{n-k}(D_\star, D_j)D_j\Bigr)A_j\Bigr\}_{j=0}^N\biggl)$$
$$ = \Gamma(\{q_{n+1}(D_\star, D_j)A_j\}_{j=0}^N) + 
\Gamma(n,\{(n+1)A_j\}_{j=0}^N)$$
$$ +
\sum_{k=1}^{n-1}
\Gamma(k,\{\textstyle\frac{(n+1)!}{k!(n+1-k)!}q_{n+1-k}(D_\star, D_j)A_j\}_{j=0}^N)$$
$$=\sum_{k=0}^n
\Gamma(k,\{\textstyle\frac{(n+1)!}{k!(n+1-k)!}q_{n+1-k}(D_\star, D_j)A_j\}_{j=0}^N).$$
Thus, \eqref{66} for $1\leq n\leq N+1$ is proved. 
Now assume that \eqref{66}  is proved for some $n$, $n\geq N+1$. 
Acting as in the case $1\leq n\leq N+1$, and taking into account that 
$\Gamma(N,\{L_j\}_{j=0}^N)\text{\bf S}=\mathbb  O$ for arbitrary operators 
$L_j$ on $\mathcal H$, we obtain that 
$$T_{(n+1)} = 
\sum_{k=0}^N
\Gamma(k,\{\textstyle\frac{n!}{k!(n-k)!}D_\star^{n-k}A_j\}_{j=0}^N)$$
$$ +
\sum_{k=0}^{N-1}
\Gamma(k+1,\{\textstyle\frac{n!}{k!(n-k)!}q_{n-k}(D_\star, D_j)A_j\}_{j=0}^N)$$
$$ +
\sum_{k=0}^N
\Gamma(k,\{\textstyle\frac{n!}{k!(n-k)!}q_{n-k}(D_\star, D_j)D_jA_j\}_{j=0}^N)$$
$$ = 
\Gamma(\{D_\star^nA_j\}_{j=0}^N) + 
\Gamma(q_n(D_\star, D_j)D_jA_j\}_{j=0}^N)$$
$$ + 
\sum_{k=1}^N
\Gamma\biggl(k,\Bigl\{\Bigl(\textstyle\frac{n!}{k!(n-k)!} D_\star^{n-k} $$
$$+
\textstyle\frac{n!}{(k-1)!(n-k+1)!}q_{n-k+1}(D_\star, D_j) +
\frac{n!}{k!(n-k)!}q_{n-k}(D_\star, D_j)D_j\Bigr)A_j\Bigr\}_{j=0}^N\biggr)$$
$$=\sum_{k=0}^N
\Gamma(k,\{\textstyle\frac{(n+1)!}{k!(n+1-k)!}q_{n+1-k}(D_\star, D_j)A_j\}_{j=0}^N)$$
(we apply \eqref{61}).

We infer from \eqref{65} and \eqref{66} that 
$$T_{(n)} = R_{(n)} + \sum_{k=0}^{n-2}
\Gamma(k,\{\textstyle\frac{n!}{k!(n-k)!}q_{n-k}(D_\star, D_j)A_j\}_{j=0}^N), 
\ \ 2\leq n\leq N+1.$$
Since $\|\Gamma\{L_j\}_{j=0}^N)\| \leq\sum_{j=0}^N\|L_j\|$ for arbitrary operators $L_j$, $0\leq j\leq N$,
 we have that 
 \begin{equation} \label{67} \begin{aligned}\|T_{(n)} - R_{(n)}\|\leq \sum_{k=0}^{n-2}
\|\Gamma(k,\{\textstyle\frac{n!}{k!(n-k)!}q_{n-k}(D_\star, D_j)A_j\}_{j=0}^N)\| \\ 
\leq\sum_{k=0}^{n-2}\sum_{j=k}^N
\textstyle\frac{n!}{k!(n-k)!}\|q_{n-k}(D_\star, D_j)\|\|A_j\|, 
\ \ 2\leq n\leq N+1. \end{aligned} \end{equation} 
If $0<\delta<1$ and $\|D_\star\|<\delta$, $\|D_j\|<\delta$, 
$0\leq j\leq N$,
then $\|q_{n-k}(D_\star, D_j)\|<(N+1)\delta$ for 
$0\leq k\leq n-2$, $2\leq n\leq N+1$.
We infer from \eqref{67}  then there exists a constant $C$ 
which depends on $N$ and $\|A_j\|$, $0\leq j\leq N$, such that 
 \begin{equation} \label{68} \|T_{(n)} - R_{(n)}\|\leq C\delta \ \ \text{ for every } n, 
\ \ 2\leq n\leq N+1,  \end{equation} 
if $0<\delta<1$ and $\|D_\star\|<\delta$, $\|D_j\|<\delta$, 
$0\leq j\leq N$.

For a polynomial $p$ put 
$$p_{1N} = \sum_{n=0}^{N+1} \widehat p(n) \chi^n \ \ \text{ and } \ \ 
p_{2N} = \sum_{n\geq N+2} \widehat p(n) \chi^n. $$
We infer from \eqref{68} that $$\|T_{(p_{1N})} - R_{(p_{1N})}\| =
\biggl\|\sum_{n=2}^{N+1} \widehat p(n)(T_{(n)} - R_{(n)})\biggr\| $$
$$\leq 
\sum_{n=2}^{N+1} |\widehat p(n)|\|T_{(n)} - R_{(n)}\| \leq
\sum_{n=2}^{N+1}\|p\|_\infty C\delta = NC\delta\|p\|_\infty,$$ 
that is, there exists a constant $C_1$ 
which depends on $N$ and $\|A_j\|$, $0\leq j\leq N$, such that
 \begin{equation} \label{69} \|T_{(p_{1N})} - R_{(p_{1N})}\|\leq C_1\delta\|p\|_\infty  \end{equation} 
for every polynomial $p$, if $0<\delta<1$ and $\|D_\star\|<\delta$, 
$\|D_j\|<\delta$, $0\leq j\leq N$.

We infer from \eqref{65}  that $R_{(p_{2N})}=\mathbb  O$ for every polynomial $p$.
From \eqref{66}  we have 
$$T_{(p_{2N})} =\sum_{n\geq N+2} \widehat p(n)T_{(n)} =
\sum_{n\geq N+2} \widehat p(n)\sum_{k=0}^N
\Gamma(k,\{\textstyle\frac{n!}{k!(n-k)!}q_{n-k}(D_\star, D_j)A_j\}_{j=0}^N
)$$
$$ =
\sum_{k=0}^N\Gamma\Bigl(k,\Bigl\{\sum_{n\geq N+2} 
\widehat p(n)\textstyle\frac{n!}{k!(n-k)!}q_{n-k}(D_\star, D_j)
A_j\Bigr\}_{j=0}^N\Bigr),$$
therefore, 
 \begin{equation} \label{610}\begin{aligned} \|T_{(p_{2N})}\|\leq
\sum_{k=0}^N\biggl\|\Gamma\Bigl(k,\Bigl\{\sum_{n\geq N+2} 
\widehat p(n)\textstyle\frac{n!}{k!(n-k)!}q_{n-k}
(D_\star, D_j)A_j\Bigr\}_{j=0}^N\Bigr)\biggr\| \\ 
\leq \sum_{k=0}^N\sum_{j=k}^N
\Bigl\|\sum_{n\geq N+2} 
\widehat p(n)
\textstyle\frac{n!}{k!(n-k)!}q_{n-k}(D_\star, D_j)\Bigr\|\|A_j\|.  \end{aligned} \end{equation} 

Denote by $\lambda_l$ and $\nu_{jl}$ the eigenvectors of 
$D_\star$ and $D_j$, $0\leq j\leq N$ (recall that $D_\star$ and $D_j$ for $0\leq j\leq N$
are simultaneously diagonalizable with respect to some orthonormal basis). 
We have  \begin{equation} \label{611} \biggl\|\sum_{n\geq N+2} 
\widehat p(n)\textstyle\frac{n!}{k!(n-k)!}q_{n-k}(D_\star, D_j)\biggr\| =
\displaystyle\sup_l \biggl|\sum_{n\geq N+2} 
\widehat p(n)\textstyle\frac{n!}{k!(n-k)!}q_{n-k}(\lambda_l, \nu_{jl})\biggr|,  \end{equation} 
where $0\leq k\leq N$. From  \eqref{610},  \eqref{611}, and Lemma \ref{lem61}   we conclude that 
there exists a constant $C$ 
which depends on $N$ and $\|A_j\|$, $0\leq j\leq N$, such that 
 \begin{equation} \label{612} \|T_{(p_{2N})}\|\leq C\delta\|p_{2N}\|_\infty \end{equation} 
for every polynomial $p$, if $0<\delta\leq 1/2$, $\|D_\star\|<\delta$ and $\|D_j\|<\delta$, 
$0\leq j\leq N$.  There exists a constant $c$ (which does not depend on $N$)
such that
$\|p_{2N}\|_\infty\leq c\log N\|p\|_\infty$ 
for every polynomial $p$, $N\geq 2$ [T, 13.4.3].
We conclude from  \eqref{612}  that 
there exists a constant $C_2$ 
which depends on $N$ and $\|A_j\|$, $0\leq j\leq N$, such that
 \begin{equation} \label{613} \|T_{(p_{2N})}\|\leq C_2\delta\|p\|_\infty \end{equation} 
for every polynomial $p$, if $0<\delta\leq 1/2$ and $\|D_\star\|<\delta$, 
$\|D_j\|<\delta$, $0\leq j\leq N$.

Finally, if $0<\delta\leq 1/2$ and $\|D_\star\|<\delta$, $\|D_j\|<\delta$, 
$0\leq j\leq N$, then from  \eqref{69}  and  \eqref{613}  
$$\|T_{(p)} - R_{(p)}\| = 
\|T_{(p_{1N})} - R_{(p_{1N})} + T_{(p_{2N})}\|$$
$$\leq 
\|T_{(p_{1N})} - R_{(p_{1N})}\| + \|T_{(p_{2N})}\|
\leq(C_1 + C_2)\delta\|p\|_\infty$$
for every polynomial $p$. 
Put $\delta = \varepsilon/(C_1 + C_2)$. The theorem is proved. \end{proof}

\begin{theorem}\label{thm63} There exist sequences of operators $\{T_N\}_N$
 and of finite Blaschke products $\{B_N\}_N$ with zeros
from $(0,1)$ such that
$\displaystyle\sup_N M_{pb}(T_N)<\infty$,
 $\displaystyle\sup_N M_{cpb}(T_N)=\infty$, 
and $T_N\approx T_{B_N}$.\end{theorem}

\begin{proof} Let $\{\alpha_{Nj}\}_{j=0}^N$, $N\geq 1$,  satisfy  \eqref{59},
and let $R_N=R_N(\{\alpha_{Nj}\}_{j=0}^N)$ be defined in  \eqref{58}. Then 
$\sup_N M_{pb}(R_N)<\infty$, and 
there exists a sequence of families of polynomials
$[p_{Nij}]_{i,j=1}^{2^{N+1}}$, $N\geq 1$,
such that 
 \begin{equation} \label{614} \| [R_{N,(p_{Nij})}]_{i,j=1}^{2^{N+1}}\|\geq 
C_N\| [p_{Nij}]_{i,j=1}^{2^{N+1}}\|_{H^\infty(\ell^2_{2^{N+1}})},
 \ \ 
\text{ where } \ C_N\to \infty,  \end{equation} 
here $R_{N,(p)}$ are defined for $R_N$ as in Theorem \ref{thm62}  
(see [DP, Theorem 3.1]).
Let a sequence $\{\varepsilon_N\}_N$ be such that $\varepsilon_N>0$ and 
 \begin{equation} \label{615}  C = \sup_N 2^{N+1}\varepsilon_N < \infty.  \end{equation} 
Applying Lemma \ref{lem41}  to $N$ and $\varepsilon_N$ and Theorem  \ref{thm62}  
to $N$, $R_N$, and $\varepsilon_N$ we obtain 
a sequence $\{\delta_N\}_N$ such that $\delta_N$ 
satisfies  the conclusion of Lemma \ref{lem41}  and Theorem \ref{thm62}   
for every $N$. Let $\{a_{Nl}\}_{l=0}^N$ and
$\{c_{Nl}\}_{l=0}^N$ be from Lemma \ref{lem51} applied to $\delta_N$.
Put $D_N = \text{\bf D}(\{a_{Nl}\}_{l=0}^N,\{c_{Nl}\}_{l=0}^N)$, and
$D_{Nj} = \text{\bf D}_j(\{a_{Nl}\}_{l=0}^N,\{c_{Nl}\}_{l=0}^N)$,
$0\leq j\leq N$, where $\text{\bf D}$ and $\text{\bf D}_j$ are defined in \eqref{52}  and  \eqref{53}, respectively. 
Denote the eigenvalues of $D_N$ by $\lambda_{Nl}$, $1\leq l\leq 2^{N+1}$.
By Lemma \ref{lem51}, $\lambda_{Nl}\in(0,\delta_N)$ and $\lambda_{Nl}\neq\lambda_{Nk}$, 
$l\neq k$. Also, $D_{Nj}$ are diagonal with respect to the standard basis in $\mathbb  C^{2^{N+1}}$, 
and the elements of $D_{Nj}$ are from $(0,\delta_N)$. 
Let $\nu_{Nl}\in(0,\delta_N)$ be such that 
$\nu_{Nl}\neq\nu_{Nk}$ for $l\neq k$,
and $\nu_{Nl}\neq\lambda_{Nk}$ for all $l$, $k$, $1\leq k,l\leq 2^{N+1}$.
Denote by
$D_{N\star}$ the  diagonal operator with respect 
to the standard basis on $\mathbb  C^{2^{N+1}}$, 
with the eigenvalues $\nu_{Nl}$.
Put $B_N=\prod_{l=1}^{2^{N+1}}b_{\lambda_{Nl}}^{N+1}b_{\nu_{Nl}}^{N+1}$ and
 $$T_N = R_N +\textstyle\bigoplus_{n=0}^N D_{N\star} \oplus\bigoplus _{n=0}^N D_N.$$
By Lemma \ref{lem43}, $T_N\approx T_{B_N}$. 

Put $T_{N0} = \oplus_{n=0}^N D_N+\text{\bf S}_{\mathbb  C^{2^{N+1}}N}$
and $T_{N1} = \oplus_{n=0}^N D_{N\star}+\text{\bf S}_{\mathbb  C^{2^{N+1}}N}^\ast$.
By Corollary \ref{cor42} applied to $T_{N0}$ and $T_{N1}$ 
there exist operators $X_{N0}$ and $X_{N1}$ such that

\noindent 
$X_{N0}T_{N0} = (\oplus_{l=1}^{2^{N+1}}T_{b_{\lambda_{Nl}}^{N+1}})X_{N0}$,
 $X_{N1}T_{N1} = (\oplus_{l=1}^{2^{N+1}}T_{b_{\nu_{Nl}}^{N+1}})X_{N1}$,
 $\|X_{N0}\|\leq 1+\varepsilon_N$, $\|X_{N0}^{-1}\|\leq 1+\varepsilon_N$,
$\|X_{N1}\|\leq 1+\varepsilon_N$, and $\|X_{N1}^{-1}\|\leq 1+\varepsilon_N$.
Therefore,  \begin{equation} \label{616}  M_{pb}(T_{N0})\leq(1+\varepsilon_N)^2 \ \ \text{ and} \ \ 
M_{pb}(T_{N1})\leq(1+\varepsilon_N)^2.  \end{equation} 
By Theorem  \ref{thm62},
 \begin{equation} \label{617} \|T_{N,(p)} - R_{N,(p)}\|\leq\varepsilon_N\|p\|_\infty 
\ \ \text{ for every polynomial }\ p.  \end{equation} 
From \eqref{616}  and \eqref{617}  we conclude that $\sup_N M_{pb}(T_N)<\infty$.

From \eqref{615}  and \eqref{617}  we have that
 \begin{equation} \label{618}\begin{aligned}\big\| [T_{N,(p_{Nij})} - R_{N,(p_{Nij})}]_{i,j=1}^{2^{N+1}}\big\|\leq 
2^{N+1}\sup_{1\leq i,j\leq 2^{N+1}}\|T_{N,(p_{Nij})} - R_{N,(p_{Nij})}\| \\
\leq
2^{N+1}\varepsilon_N\sup_{1\leq i,j\leq 2^{N+1}}\|p_{Nij}\|_\infty
\leq C\big\| [p_{Nij}]_{i,j=1}^{2^{N+1}}\big\|_{H^\infty(\ell^2_{2^{N+1}})}. \end{aligned}\end{equation}  

If $C_N>C$, then from \eqref{614} and \eqref{618}
$$\| [T_{N,(p_{Nij})}]_{i,j=1}^{2^{N+1}}\|\geq
\| [R_{N,(p_{Nij})}]_{i,j=1}^{2^{N+1}}\|-
\| [T_{N,(p_{Nij})} - R_{N,(p_{Nij})}]_{i,j=1}^{2^{N+1}}\|$$
$$\geq
(C_N - C) \| [p_{Nij}]_{i,j=1}^{2^{N+1}}\|_{H^\infty(\ell^2_{2^{N+1}})},$$
therefore, $\sup_N M_{cpb}(T_N)=\infty$. \end{proof}

\begin{corollary}\label{cor64} There exist an operator $T$ 
and a Blashcke product $B$ with zeros from $(0,1)$ 
such that $T$ is polynomially bounded, $T$ is not similar 
to a contraction, and $T\sim T_B$, where $T_B$ is defined in \eqref{17}.
\end{corollary}

\begin{proof} Let  $\{T_N\}_N$ and $\{B_N\}_N$ be the
sequences of operators  and of finite Blaschke products  with zeros
from $(0,1)$ from Theorem \ref{thm63}. By Lemma \ref{lem31},  there exists a sequence
$\{w_N\}_N\subset(0,1)$ such that $B=\prod_N\zeta_N B_N\circ\beta_{w_N}$
converges, and $ B_N\circ\beta_{w_N}$ are pairwise coprime.

Put $T=\oplus_N\beta_{w_N}(T_N)$. By Lemma \ref{lem32},
$\beta_{w_N}(T_N)\approx \beta_{w_N}(T_{B_N})\cong T_{B_N\circ\beta_{w_N}}$,
therefore, $T\sim T_B$. 
By Corollary \ref{cor12}, $T$ is polynomially bounded, and $T$ is not similar 
to a contraction. \end{proof}

Applying Corollary \ref{cor23} to the operator $T$ from Corollary \ref{cor64}, 
one can obtain a polynomially bounded operator $\text{\bf T}$ such that
$\text{\bf T}\prec S$, and $\text{\bf T}$ is not similar 
to a contraction. Denote by $X$ and $Y$ quasiaffinities 
such that $YT=T_BY$ and $XT_B = TX$. By [SFBK, Theorem X.2.10], 
there exists a function $\varphi\in H^\infty$ such that 
$YX=\varphi(T_B)$.
The constructions of the operators $R_N$ from [DP] and 
 of the operator $T$ from Corollary  \ref{cor64} are explicit, therefore, 
the quasiaffinities $X$ and $Y$ and a function $\varphi$
can be computed, but the author cannot do it. Also, 
since $\text{\bf T}$ is polynomially bounded and
$\text{\bf T}\prec S$, it follows from [BP] that
$\text{\bf T}\sim S$ if and only if 
$\mu_{\text{\bf T}}=1$. But the author cannot
 compute $\mu_{\text{\bf T}}$.

\section{The construction of quasisimilar operators such that the product 
of intertwining quasiaffinities is an outer function of operators}

\begin{theorem}\label{thm71} There exist an operator 
$T\colon\mathcal H\to\mathcal H$,  
a Blashcke product $B$ with zeros from $(0,1)$, 
an outer function $g\in H^\infty$, and quasiaffinities
$X\colon\mathcal H\to\mathcal K_B$, $Y\colon\mathcal K_B\to\mathcal H$, 
such that $T$ is polynomially bounded, $T$ is not similar 
to a contraction, $XT=T_BX$, $YT_B=TY$, and $XY=g(T_B)$,
where $T_B$ is defined in \eqref{17}.
\end{theorem}

\begin{proof} Let $\{T_N\}_N$ and $\{B_N\}_N$ be
the sequences of operators 
 and of finite Blaschke products with zeros
from $(0,1)$ from Theorem \ref{thm63}, respectively. 
Let $C>0$ be fixed. Denote by $\mathcal H_N$ 
the finite dimensional spaces on which $T_N$ acts. 
There exists invertible
operators $X_N\colon\mathcal H_N\to\mathcal K_{B_N}$, 
$Y_N\colon\mathcal K_{B_N}\to\mathcal H_N$ such that
$X_NT_N=T_{B_N}X_N$, $Y_NT_{B_N}=T_NY_N$,
$\|X_N\|\leq C$, $\|Y_N\|\leq C$. 
By [SFBK, Theorem X.2.10], there exist functions $\varphi_N\in H^\infty$
such that $X_NY_N=\varphi_N(T_{B_N})$. 
Note that 
 $\varphi_N(\lambda)\neq 0$ for every $\lambda\in\mathbb  D$
such that $B_N(\lambda)= 0$, for every index $N$.

Let $g$ be from Lemma \ref{lem34}. Applying Theorem \ref{thm38} to 
$C$, $g$, sequences of  $B_N$ and of $\varphi_N$ we obtain 
$\delta>0$ and sequences of $w_N\in(0,1)$, of
$\zeta_N\in\mathbb  T$ and of $\psi_N \in H^\infty$
 which satisfy the conclusion of Theorem \ref{thm38}.
Put $$B=\prod_N\zeta_NB_N\circ\beta_{w_N} \ \ \text{ and } 
\ \ T=\oplus_N\beta_{w_N}(T_N).$$ 
By Theorem \ref{thm63}  and Corollary  \ref{cor12}, $T$ is polynomially bounded, 
and $T$ is not similar 
to a contraction. 

Put $U_N=U_{w_N}|_{\mathcal K_{B_N\circ\beta_{w_N}}}
\colon\mathcal K_{B_N\circ\beta_{w_N}}\to\mathcal K_{B_N}$, where $U_w$ is defined in 
Lemma \ref{lem32}. Then $U_N$ is unitary, and 
$U_NT_{B_N\circ\beta_{w_N}} = \beta_{w_N}(T_{B_N})U_N$.
Put $$X_{1N} = U_N^{-1}(\psi_N\circ\beta_{w_N})(T_{B_N})X_N \ \
\text{ and } \ \ Y_{1N} = Y_NU_N.$$
Then $$X_{1N} \beta_{w_N}(T_N) = T_{B_N\circ\beta_{w_N}}X_{1N}, \ \
Y_{1N}T_{B_N\circ\beta_{w_N}} = \beta_{w_N}(T_N)Y_{1N}, $$
$$ X_{1N}Y_{1N}=g(T_{B_N\circ\beta_{w_N}}), \ \ \
\|X_{1N}\|\leq C^2, \ \ \ \|Y_{1N}\|\leq C, $$
and $X_{1N}$ and $Y_{1N}$ are quasiaffinities.
Indeed, 
$X_N\beta_{w_N}(T_N)=\beta_{w_N}(T_{B_N})X_N$ and 
$Y_N\beta_{w_N}(T_{B_N}) = \beta_{w_N}(T_N)Y_N$.
Therefore, 
$$X_{1N} \beta_{w_N}(T_N) = 
U_N^{-1}\psi_N(\beta_{w_N}(T_{B_N}))\beta_{w_N}(T_{B_N})X_N $$
$$ = 
 U_N^{-1}\beta_{w_N}(T_{B_N})\psi_N(\beta_{w_N}(T_{B_N}))X_N
=
T_{B_N\circ\beta_{w_N}}U_N^{-1}\psi_N(\beta_{w_N}(T_{B_N}))X_N $$
$$ =
T_{B_N\circ\beta_{w_N}}X_{1N}.$$
Also, $$Y_{1N}T_{B_N\circ\beta_{w_N}} = Y_NU_NT_{B_N\circ\beta_{w_N}}
=Y_N\beta_{w_N}(T_{B_N})U_N 
=\beta_{w_N}(T_N)Y_{1N}.$$
Furthermore, 
$$(\psi_N\circ\beta_{w_N})(T_{B_N})X_N Y_N =
(\psi_N\circ\beta_{w_N})(T_{B_N})\varphi_N(T_{B_N})=
(\varphi_N\cdot\psi_N\circ\beta_{w_N})(T_{B_N}).$$
 Since 
$g-\psi_N\cdot\varphi_N\circ\beta_{w_N}\in (B_N\circ\beta_{w_N}) H^\infty$, 
we have that $g\circ\beta_{w_N}-\varphi_N\cdot\psi_N\circ\beta_{w_N}
\in B_N H^\infty$, 
therefore, 
$(\varphi_N\cdot\psi_N\circ\beta_{w_N})(T_{B_N})= 
(g\circ\beta_{w_N})(T_{B_N}) = g(\beta_{w_N}(T_{B_N}))$.
Thus, $X_{1N}Y_{1N} = U_N^{-1}g(\beta_{w_N}(T_{B_N}))U_N = 
g(T_{B_N\circ\beta_{w_N}})$.
By Nehari's theorem, 
$$\|X_{1N}\| \leq
\|(\psi_N\circ\beta_{w_N})(T_{B_N})\|\|X_N\|\leq
C\operatorname{dist}(\psi_N\circ\beta_{w_N}, B_NH^\infty)$$
$$=
C\operatorname{dist}(\psi_N, (B_N\circ\beta_{w_N}) H^\infty)\leq C^2. $$
Clearly,  $\|Y_{1N}\| =\|Y_N\| \leq C$. Since 
$X_{1N}Y_{1N}=g(T_{B_N\circ\beta_{w_N}})$ and $g$ is outer, 
$X_{1N}Y_{1N}$ is a quasiaffinity, and, since $Y_{1N}$ is invertible,
we conclude that $X_{1N}$ is a quasiaffinity.

We have $\mathcal K_B= \bigvee_N\mathcal K_{B_N\circ\beta_{w_N}}$.
Define 
a linear mapping $J\colon\mathcal K_B\to\oplus_N\mathcal K_{B_N\circ\beta_{w_N}}$
by the formula $J\sum_N x_N = \oplus_N x_N$, 
where $x_N\in\mathcal K_{B_N\circ\beta_{w_N}}$,
the cardinality of the set of $N$ such that $x_N\neq 0$ is finite.
Since $\{B_N\circ\beta_{w_N}\}_N$ satisfy the condition  \eqref{36}, 
$J$ is expanded on $\mathcal K_B$, is bounded and invertible
(its inverse is bounded) [Ni, Theorems II.C.3.1.4 and II.C.3.2.14].  
Since the spaces $\mathcal K_{B_N\circ\beta_{w_N}}$ are 
invariant for $T_B^\ast$, it is easy to see that 
$JT_B^\ast = (\oplus_N T_{B_N\circ\beta_{w_N}})^\ast J $.

Define the following operators:
$$Z\colon\oplus_N\mathcal H_N\to \mathcal K_B, \ \ \ Z = J^{-1}(\oplus_NY_{1N}^\ast),$$
$$W\colon\mathcal K_B\to\oplus_N\mathcal H_N, \ \ \ W = (\oplus_NX_{1N}^\ast)J.$$
It is easy to see that $Z$ and $W$ are quasiaffinities.
Also,  \begin{equation} \label{71} ZT^\ast = T_B^\ast Z, \ \ WT_B^\ast = T^\ast W,
\ \  \text{ and } \ \ ZW = g_\ast(T_B^\ast),  \end{equation}  
where $g_\ast(z)=\overline{g(\overline z)}$, $z\in\mathbb  D$. Indeed, 
$$ ZT^\ast = 
J^{-1}(\oplus_NY_{1N}^\ast)\bigl(\oplus_N(\beta_{w_N}(T_N))\bigr)^\ast
=J^{-1}\bigl(\oplus_NY_{1N}^\ast(\beta_{w_N}(T_N))^\ast\bigr)$$
$$
=J^{-1}(\oplus_N T_{B_N\circ\beta_{w_N}})^\ast(\oplus_NY_{1N}^\ast)=
T_B^\ast J^{-1}(\oplus_NY_{1N}^\ast)=T_B^\ast Z.$$
Also, 
$$ WT_B^\ast =(\oplus_NX_{1N}^\ast)J T_B^\ast =
(\oplus_NX_{1N}^\ast)(\oplus_N T_{B_N\circ\beta_{w_N}})^\ast J$$
$$ =\bigl(\oplus_N(\beta_{w_N}(T_N))^\ast\bigr)(\oplus_NX_{1N}^\ast) J 
=T^\ast W.$$
Furthermore, 
$$ZW =J^{-1}(\oplus_NY_{1N}^\ast)(\oplus_NX_{1N}^\ast)J =
J^{-1}\bigl(\oplus_N(X_{1N}Y_{1N})^\ast\bigr)J $$
$$= 
J^{-1}g_\ast\bigl(\oplus_N(T_{B_N\circ\beta_{w_N}})^\ast\bigr)J = g_\ast(T_B^\ast).$$
The conclusion of the theorem follows from \eqref{71}. \end{proof}

\section*{References}


 [AT] T. Ando and K. Takahashi, On operators with unitary $\rho$-dilations, 
{\it Ann. Polon. Math.}, 
{\bf 66} (1997), 11--14.

[B] C. Badea,
Perturbations of operators similar to contractions and the commutator equation. 
{\it Stud. Math.}, {\bf 150} (2002), No.3, 273-293.

[BP] H. Bercovici  and B. Prunaru,   Quasiaffine transforms 
of polynomially bounded operators, 
{\it Arch. Math. (Basel)}, {\bf 71} (1998),  384--387.

[C] G. Cassier, Generalized Toeplitz operators, restriction to invariant subspaces and similarity problems. 
{\it J. Oper. Theory}, {\bf 53} (2005), No. 1, 49-89.

[DP] K. R. Davidson and V. I. Paulsen, Polynomially bounded operators,
 {\it J. reine angew. Math}, {\bf 487} (1997), 153-170.

 [H] P. R. Halmos, Ten problems in Hilbert space, 
{\it Bull. Am. Math. Soc.}, {\bf 76} (1970), 887-933.

[K]  L. K\'erchy, Quasianalytic polynomially bounded operators, {\it Operator Theory: the State of the Art}, Theta, Bucharest, 2016, 75--101.

  [M] W. Mlak, Algebraic polynomially bounded operators, 
{\it Ann. Polon. Math.}, {\bf 29} (1974), 133--139.

[Ni] N. K. Nikolski, {\it  Operators, functions, and systems: an easy reading.
 Volume I:
Hardy, Hankel, and Toeplitz, Volume II: Model operators and systems}, Math. Surveys and Monographs 
{\bf 92}, AMS, 2002.  

  [No] E. A. Nordgren, The ring $N^+$ is not adequate, {\it Acta Sci. Math. (Szeged)},
  {\bf 36} (1974), 203--204.

[Pa] V. I. Paulsen, Every completely polynomially bounded operator is similar to a contraction, 
{\it J. Funct. Anal.}, {\bf 55} (1984), 1-17.

[Pe] V. V. Peller, {\it Hankel operators and their applications},
Springer Monographs in Math. New York,
 NY, Springer, 2003.

[Pi] G. Pisier, A polynomially bounded operator on Hilbert space which 
is not similar to a contraction,
{\it J. Am. Math. Soc.}, {\bf 10} (1997), No.2, 351-369. 

[SFBK] B. Sz.-Nagy,  C. Foias, H. Bercovici and L. K\'erchy, {\it Harmonic analysis of operators on 
Hilbert space},  Springer, New York, 2010.

[T] E. C. Titchmarsh, {\it The theory of functions}. 2nd ed. London, 
Oxford University Press, 1975.
\end{document}